\documentclass[a4paper]{article}

\usepackage{graphicx, epsfig}
\usepackage{enumerate}
\usepackage{subcaption}
\usepackage{amssymb, amsmath, amsfonts}
\usepackage{xcolor}
\usepackage{bbm}
\usepackage{bm}
\usepackage{setspace}
\usepackage[margin=1in]{geometry}
\usepackage{makeidx}
\usepackage{hyperref}
\usepackage{authblk}
\usepackage[lined,boxed,commentsnumbered,noend]{algorithm2e}

\graphicspath{{images/}}

\usepackage{amsthm, amssymb, amsmath, amsfonts}

\newcommand{\rd}{\mathbb{R}^d}
\newcommand{\rk}{\mathbb{R}^k}

\newcommand{\rr}{\mathbb{R}}

\newcommand{\E}[2][]{\mathbb{E}_{#1} \left[ #2 \right]}
\newcommand{\conE}[3][]{\E[#1]{#2 | #3}}
\newcommand{\Prob}[1]{\mathbb{P}\left(#1\right)}

\newcommand{\norm}[1][\cdot]{\left\| \kern.05em #1 \kern.05em \right\|}
\newcommand{\var}{\beta\operatorname{-VaR}}
\newcommand{\cvar}{\beta\operatorname{-CVaR}}

\newcommand{\conic}[1]{\operatorname{conic}\left( #1 \right)}

\newcommand{\interior}[1]{\operatorname{int}\left(#1\right)}

\newcommand{\minimize}[1][]{\underset{#1}{\operatorname{minimize}}\ }

\newcommand{\rv}{\bm{\xi}}
\newcommand{\orv}{\xi}
\newcommand{\rvsup}{\Xi}
\newcommand{\arv}{\bm{\tilde{\xi}}}

\newcommand{\riskregion}{\mathcal{R}}

\newcommand{\conwpo}{\overset{\text{w.p.1}}{\to}}

\newtheorem{theorem}{Theorem}[section]
\newtheorem{corollary}[theorem]{Corollary}
\newtheorem{lemma}[theorem]{Lemma}
\newtheorem{proposition}[theorem]{Proposition}
\newtheorem{definition}[theorem]{Definition}

\newtheorem{example}[theorem]{Example}

\begin{document}

\title{Problem-driven scenario generation: an analytical approach for stochastic programs with tail risk measure}
\author[*]{Jamie Fairbrother}
\author[*]{Amanda Turner} 
\author[**]{Stein W. Wallace}
\affil[*]{STOR-i Centre for Doctoral Training, Lancaster University. United Kingdom}
\affil[**]{Department of Business and Management Science, Norwegian School of Economics. Norway}

\maketitle

\begin{abstract}
Scenario generation is the construction of a discrete random
vector to represent parameters of uncertain values in 
a stochastic program. Most approaches to scenario generation
are \emph{distribution-driven}, that is, they attempt to construct
a random vector which captures well in a probabilistic sense
the uncertainty. On the other hand, a \emph{problem-driven}
approach may be able to exploit the structure of a problem
to provide a more concise representation of the uncertainty.

In this paper we propose an analytic approach to problem-driven scenario generation.
This approach applies to stochastic programs where a tail risk measure,
such as conditional value-at-risk, is applied to a loss function. Since tail risk measures only depend on
the upper tail of a distribution, standard methods of scenario
generation, which typically spread their scenarios evenly across
the support of the random vector, struggle to adequately
represent tail risk. Our scenario generation approach works
by targeting the construction of scenarios in areas of the distribution corresponding to the
tails of the loss distributions. We provide conditions under which our approach is consistent
with sampling, and as proof-of-concept demonstrate how our approach
could be applied to two classes of problem, namely network design and portfolio selection.
Numerical tests on the portfolio selection problem demonstrate
that our approach yields better and more stable solutions compared
to standard Monte Carlo sampling.
\end{abstract}

\section{Introduction}
\label{sec:tailrisk-intro}

Stochastic programming is a tool for making decisions under
uncertainty. Under this modeling paradigm, uncertain parameters are
modeled as a random vector, and one attempts to minimize (or
maximize) the expectation or risk measure of some loss function which
depends on the initial decision. However, what distinguishes
stochastic programming from other stochastic modeling approaches is
its ability to explicitly model future decisions based on outcomes of
stochastic parameters and initial decisions, and the associated costs
of these future decisions.  The power and flexibility of the
stochastic programming approach comes at a price: stochastic programs
are usually analytically intractable, and often not susceptible to solution
techniques for deterministic programs.

Typically, a stochastic program can only be solved when it is
\emph{scenario-based}, that is when the random vector for the problem
has a finite discrete distribution. For example,
stochastic linear programs become large-scale linear programs when the
underlying random vector is discrete. In the stochastic
programming literature, the mass points of this random vector are
referred to as \emph{scenarios}, the discrete distribution as the
\emph{scenario set} and the construction of this set as \emph{scenario
generation}. Scenario generation can consist of discretizing a
continuous probability distribution, or directly modeling the
uncertain quantities as discrete random variables. The more
scenarios in a set, the more computational power that is required
to solve the problem. The key issue of scenario generation is therefore how
to represent the uncertainty to ensure that the solution to the
problem is reliable, while keeping the number of scenarios low
so that the problem is computationally tractable.

A common approach to scenario generation is to fit a statistical
model to the uncertain problem parameters and then
generate a random sample from this for the scenario set. This has desirable asymptotic properties \cite{MR94k:90055,Shapiro03b}, but may
require large sample sizes to ensure the reliability of the solutions
it yields. This can be mitigated somewhat by using variance reduction
techniques such as stratified sampling and importance sampling
\cite{LinderothEA06}. Sampling also has the advantage that it can be
used to construct confidence intervals on the true solution value
\cite{mak_morton_wood_99}. Another approach  is to
construct a scenario set whose distance from the true distribution,
with respect to some probability metric, is small \cite{Pflug01,HeRo09,DupacovaEA03}.  
These approaches tend to yield better and much more
stable solutions to stochastic programs than does sampling.

A characteristic of these approaches to scenario generation is that
they are \emph{distribution-driven}; that is, they only aim to approximate
a distribution and are divorced from the stochastic program for which they
are producing scenarios.
By exploiting the structure of a problem, it may be possible to find a more parsimonious
representation of the uncertainty. Note that such a
\emph{problem-driven} approach may not yield a discrete distribution
which is close to the true distribution in a probabilistic sense; the
aim is only to find a discrete distribution which yields a high
quality solution to our problem.

Stochastic programs often have the objective of minimizing the expectation of a loss function. This is particularly appropriate when the initial decision represents a strategic decision that is going to be used again and again, and
individual large losses do not matter in the long term.
For example, in a stochastic facility location problem (e.g. see \cite{bienek2015})
the locations of several facilities must be chosen subject to the unknown
demands of customers in a way which minimizes fixed investment costs,
and future distribution costs.
In other cases, the decision may be only used once or a few times, and the
occurrence of large losses may have serious consequences such as bankruptcy. 
This is characteristic of the portfolio selection problem \cite{Mark52}
studied in detail in the latter part of this paper.
In this latter case, minimizing the expectation alone is not appropriate as this does
not necessarily mitigate against the possibility of large losses. One
possible remedy is to use a \emph{risk measure} which penalises in
some way the likelihood and severity of potential large losses. 

In this paper we are interested in stochastic programs which use
\emph{tail risk measures}.  A precise definition of a tail-risk
measure will be given in Section \ref{sec:tailrisk-tailrisk} but for
now, one can think of a tail risk measure as a function of a random
variable which only depends on the upper tail of its distribution
function. Tail risk measures are useful as they summarize the extent
of potential losses in the worst possible outcomes. Examples of tail
risk measures include the Value-at-Risk (VaR) \cite{Jo96} and the
Conditional Value-at-Risk (CVaR) \cite{Rockafellar00}, both of which
are commonly used in financial contexts. Although the methodology
developed in this paper can in principle be applied to any loss function,
in this work we are mainly interested in loss functions which arise in one and two-stage stochastic programs.

Distribution-driven scenario generation methods are particularly
problematic for stochastic programs involving tail risk measures.
This is because these methods tend to spread their scenarios evenly across
the support of distribution and so struggle to adequately represent the tail risk without using
a potentially prohibitively large number of scenarios.

In this paper, we propose an analytic problem-driven approach to
scenario generation applicable to stochastic programs which use tail
risk measures of a form made precise in Section~\ref{sec:tailrisk-tailrisk}.
We observe that the value of a tail risk measure depends only on scenarios
confined to an area of the distribution that we call the \emph{risk region}.
This means that all scenarios that are not in the risk
region can be aggregated into a single point.  
By concentrating scenarios in the risk region, we can calculate the
value of a tail risk measure more accurately.

Given a risk region for a problem, we propose a simple algorithm for
generating scenarios which we call \emph{aggregation sampling}. This
algorithm takes samples from the random vector until a specified
number of samples in the risk region have been produced, and all other
scenarios are aggregated into a single scenario. We provide and give
proofs of conditions under which this method is asymptotically
consistent with standard Monte Carlo sampling.

In general, finding a risk region is difficult as it is determined by
the loss function, problem constraints and the distribution of the
uncertain parameters. Therefore, we derive risk regions for two
classes of problem as a proof-of-concept of our methodology. The first
class of problems are those with monotonic loss functions which,
as will be shown, occur
naturally in the context of network design. The second class are
portfolio selection problems. For both types of risk regions we run
numerical tests which demonstrate that our methodology
yields better quality solutions and with greater reliability than
standard Monte Carlo sampling.

This paper is organized as follows: in Section~\ref{sec:related-work}
we discuss related work; in
Section~\ref{sec:tailrisk-tailrisk} we define tail risk measures and
their associated risk regions; in
Section~\ref{sec:tailrisk-aggsampling} we discuss how these risk
regions can be exploited for the purposes of scenario generation;
in Section~\ref{sec:tailrisk-conv-aggr-sampl} we
prove that our scenario generation method is consistent with 
standard Monte Carlo sampling;
in Sections~\ref{sec:monotonic} and \ref{sec:tailrisk-portfolio} we
derive risk regions for the two classes of problems described above;
in Section~\ref{sec:tailrisk-numtests} we present numerical tests;
finally in Section~\ref{sec:tailrisk-conclusions} we summarize our
results and make some concluding remarks.

\paragraph{Notation}

Throughout this paper random variables and vectors are represented by bold (mainly Greek) letters: $\bm{\theta},\ \bm{\xi},\ \bm{\zeta}$ and outcomes of these are represented by the corresponding non-bold letters: $\theta,\ \xi,\ \zeta$.
Inequalities used with vectors and matrices always apply
component-wise. $\norm$ represents the standard Euclidean norm.

\section{Related Work}
\label{sec:related-work}

There are relatively few cases of problem-driven scenario generation in
the literature. The earliest example of which we are aware
is the importance sampling approach of \cite{dantzig1990parallel} which
constructs a sampler from the loss function. Importance sampling 
has been used more recently for scenario generation for problems which, like our own,
concern rare events. In \cite{KozmikMorton2015} an importance sampling
scheme is used for a multistage problem involving the CVaR risk measure.
In \cite{barrera2016chance}, an importance sampling approach is proposed for chance-constrained
stochastic programs where the permitted probabilities of constraint
violation are very small. 

There is an interesting connection between problem-driven scenario
generation and distributionally robust optimization \cite{Zackova1966,Dupacova11,WiesemannKuhnSim14}. 
In distributionally
robust optimization, the distribution of the random variables
in a stochastic program is itself uncertain, and one must optimize
for the worst-case distribution. Solving a distributionally robust
optimization problem thus involves finding, at least implicitly, the worst-case distribution
or scenario set for given objective and constraints. In this
sense, distributionally robust optimization
could be considered as a problem-driven scenario generation method.
Of particular relevance for this work, the paper \cite{VinhDoan2015}
solves a distributionally robust portfolio selection problem involving
the CVaR risk measure where the distribution of asset returns has specified discrete marginals, but
unknown joint distribution.

The idea that in stochastic programs with tail risk measures
some scenarios do not contribute to the calculation of the tail-risk
measure was also exploited in \cite{BertrandMinguez12}.
However, they propose a solution algorithm rather than
a method of scenario generation. Their approach is
to iteratively solve the problem with a subset of scenarios, identify
the scenarios which have loss in the tail, update their scenario set
appropriately and resolve, until the true solution has been found.
Their method has the benefit that it is not distribution dependent.
On the other hand, their method works for only the $\cvar$ risk measure,
while our approach works in principle for any tail risk measure.

\section{Tail risk measures and risk regions}
\label{sec:tailrisk-tailrisk}

In this section we present the core theory to our scenario generation
methodology. Specifically, in Section~\ref{sec:tailrisk-tailrisk-rv}
we formally define tail-risk measures of random variables and in
Section~\ref{sec:tailrisk-tailrisk-opt} we define risk regions and
present some key results related to these.

 \subsection{Tail risk of random variables}
\label{sec:tailrisk-tailrisk-rv}

In our set-up we suppose we have some random variable representing an
uncertain loss. For our purposes, we take a risk measure to be any function of a random variable. The following formal
definition is adapted from \cite{tasche2002expected}.

\begin{definition}[Risk Measure]
  \label{def:risk-measure}
  Let $(\Omega, \mathbb{P})$ be a probability space,
  and $\Theta$ be the set of measurable real-valued random variables on $(\Omega, \mathbb{P})$.
  Then, a risk measure is some function $\rho: \Theta \rightarrow \rr \cup \{\infty\}$.
\end{definition}

For a risk measure to be useful, it should in some way penalize potential large losses.
For example, in the classical Markowitz problem
\cite{Mark52}, one aims to minimize the variance of the return of a portfolio.
By choosing a portfolio with a low variance, we reduce the
probability of larges losses as a direct consequence of Chebyshev's
inequality (see for instance \cite{BillingsleyMeasure95}).
Various criteria for risk measures have been proposed; in
\cite{Artzner98} a \emph{coherent risk measure} is defined to
be a risk measure which satisfies axioms such as positive
homogeneity and subadditivity; another perhaps desirable criterion for
risk measures is that the risk measure is consistent with
respect to first and second order stochastic dominance, see
\cite{MR1922754} for instance.

Besides not satisfying some of the above criteria, 
a major drawback with using variance as a measure is that
it penalizes all large deviations from the mean, that is, it penalizes
large profits as well as large losses. This motivates the idea of using risk measures which
depend only on the upper tail of the loss distribution. To formalize
this idea, we first recall the definition of
quantile function.

\begin{definition}[Quantile Function]
  Suppose $\bm{\theta}$ is a random variable with distribution function $F_{\bm{\theta}}$. Then 
  the generalized inverse distribution function, or \emph{quantile function}
  is defined as follows:
  \begin{align*}
    F^{-1}_{\bm{\theta}} : (0, 1] &\rightarrow \rr\cup\{\infty\}\\
    \beta &\mapsto \inf\{ x\in\rr: F_{\bm{\theta}}(x) \geq \beta \}.
  \end{align*}
  We refer to the quantile function evaluated at $\beta$,
  $F_{\bm{\theta}}^{-1}(\beta)$, as the \emph{$\beta$-quantile}.
\end{definition}

The $\beta$-quantile can be interpreted as the
smallest value for which the distribution function is greater
than or equal to $\beta$. The $\beta$-tail of a distribution
is the restriction of the distribution function
to values equal to or above the $\beta$-quantile. In the context
of risk management, we typically have $0.9 \leq \beta < 1.0$.
The following definition says that a tail risk measure
is a risk measure that only depends on the $\beta$-tail
of a distribution.

\begin{definition}[Tail Risk Measure]
  Let $\rho_\beta: \Theta \rightarrow \rr \cup \{\infty\}$ be a risk measure per Definition~\ref{def:risk-measure}.
  Then $\rho_\beta$ is a $\beta$-tail risk measure if $\rho_\beta(\bm{\theta})$ depends
  only on the restriction of quantile function of $\bm{\theta}$ above $\beta$,
  in the sense that if $\bm{\theta}$ and $\bm{\tilde{\theta}}$ are random variables with
  $\mathrel{F_{\bm{\theta}}^{-1}|_{[\beta,1]}}= \mathrel{F_{\bm{\tilde{\theta}}}^{-1}|_{[\beta,1]}}$ then $\rho_\beta(\bm{\theta}) = \rho_\beta(\bm{\tilde{\theta}})$.
\end{definition}

To show that $\rho_\beta$ is a $\beta$-tail risk measure, we must show that  $\rho_\beta(\bm{\theta})$ can be written as a function of the quantile function above or equal to $\beta$. Two very popular tail risk measures are the value-at-risk \cite{Jo96} and the conditional value-at-risk \cite{RockUry02}:

\begin{example}[Value at risk]
  Let $\bm{\theta}$ be a random variable, and $0< \beta < 1$. Then, the $\beta-$VaR for $\bm{\theta}$ is defined to be the $\beta$-quantile of $\bm{\theta}$:
  \begin{equation*}
    \var(\bm{\theta}) := F_{\bm{\theta}}^{-1}(\beta).
  \end{equation*}
\end{example}

\begin{example}[Conditional value at risk]
  Let $\bm{\theta}$ be a random variable, and $0 < \beta < 1$. 
  The following alternative characterization of $\cvar$
  \cite{acerbi2002coherence} shows directly that it is
  a $\beta$-tail risk measure.
  \begin{equation*}
    \cvar(\bm{\theta}) = \frac{1}{1-\beta}\int_{\beta}^1 F^{-1}_{\bm{\theta}}(u)\ du.
  \end{equation*}
  Note that in the case that $\bm{\theta}$ is a continuous random variable, the $\cvar$ is the conditional expectation of the random variable above its $\beta$-quantile (e.g. see \cite{RockUry02}).
\end{example}

The observation that we exploit for this work is that very
different random variables will have the same $\beta$-tail risk
measure as long as their $\beta$-tails are the same. 

When showing that two distributions have the same $\beta$-tails, it is
convenient to use distribution functions rather than quantile functions.
The following result gives conditions which ensure that the $\beta$-tails
of two distributions are the same. We will make use of these in proofs 
later in this paper.

\begin{lemma}
  \label{lem:quantile-beta-tail}
  Suppose that $\bm{\theta}$ and $\bm{\tilde{\theta}}$ are random
  variables such that one of the two following conditions hold:
  \begin{enumerate}[(i)]
  \item $F_{\bm{\tilde{\theta}}}(\theta) = F_{\bm{\theta}}(\theta)$ for all $\theta \geq F_{\bm{\theta}}^{-1}(\beta)$ and $F_{\bm{\tilde{\theta}}}(\theta) < \beta$ for all $\theta < F_{\bm{\theta}}^{-1}(\beta)$.
  \item $F_{\bm{\tilde{\theta}}}(\theta) = F_{\bm{\theta}}(\theta)$ for all $\theta \geq L$ for some $L < F_{\bm{\theta}}^{-1}(\beta)$.
  \end{enumerate}
  Then, $F_{\bm{\tilde{\theta}}}^{-1}(u) = F_{\bm{\theta}}^{-1}(u)$ for all $u \geq \beta$.
\end{lemma}

\begin{proof}
  We first prove that condition (i) implies that the $\beta$-tails are the same.
  Since $F_{\bm{\tilde{\theta}}}(\theta) = F_{\bm{\theta}}(\theta) \geq \beta$ for all $\theta \geq F_{\bm{\theta}}^{-1}(\beta)$,
  we must have $F_{\bm{\tilde{\theta}}}^{-1}(\beta) \leq F_{\bm{\theta}}^{-1}(\beta)$. Also, given $F_{\bm{\tilde{\theta}}}(\theta) < \beta$
  for all $\theta < F_{\bm{\theta}}^{-1}(\beta)$ we must have $F_{\bm{\tilde{\theta}}}^{-1}(\beta) \geq F_{\bm{\theta}}^{-1}(\beta)$ and so
  $F_{\bm{\tilde{\theta}}}^{-1}(\beta) = F_{\bm{\theta}}^{-1}(\beta)$.
  
  Now suppose $u \geq \beta$. Then,
  \begin{align*}
    F_{\bm{\tilde{\theta}}}^{-1}(u) &= \inf\{ \theta\in\rr :\ F_{\bm{\tilde{\theta}}}(\theta) \geq u\}\\
    &= \inf\{ \theta \geq F_{\bm{\tilde{\theta}}}^{-1}(\beta) : F_{\bm{\tilde{\theta}}}(\theta) \geq u\}\\
    &= \inf\{ \theta \geq F_{\bm{\theta}}^{-1}(\beta) : F_{\bm{\theta}}(\theta) \geq u\}\\
    &= \inf\{ \theta\in\rr :\ F_{\bm{\theta}}(\theta) \geq u\}\\
    &= F_{\bm{\theta}}^{-1}(u)
  \end{align*}
  where the second and fourth lines follow from the fact that quantile functions are non-decreasing.

  In the case condition (ii) holds, we have for $L < \theta < F_{\bm{\theta}}^{-1}(\beta)$ that $F_{\bm{\tilde{\theta}}}(\theta) = F_{\bm{\theta}}(\theta) < \beta$,
  and since distribution functions are non-decreasing this means that $F_{\bm{\tilde{\theta}}}(\theta) < \beta$ for all $\theta < F_{\bm{\theta}}^{-1}(\beta)$.
  The result now follows by application of condition (i).
\qed
\end{proof}

\subsection{Risk regions}
\label{sec:tailrisk-tailrisk-opt}

In this paper we are primarily interested in problems of the following form:
\begin{equation}
  \label{eq:tail-risk-sp-general}
  \minimize[x\in\mathcal{X}] \rho_{\beta}(f(x,\rv))
\end{equation}
where $\mathcal{X} \subseteq \rk$ is a deterministic set of feasible
decisions, $\rv\in\rvsup\subseteq\rd$ is a random vector defined on a probability space $(\Omega, \mathbb{P})$, the set $\rvsup$ is convex, $f: \mathcal{X}\times\rvsup \rightarrow \rr$ is a \emph{loss function},
and $\rho_{\beta}$ is a \emph{tail risk measure}.

In order to solve these problems accurately, we need to be able to approximate well the tail risk measure of our the loss function $f(x, \rv)$ for all feasible decisions $x\in\mathcal{X}$.

To avoid repeated use of cumbersome notation we introduce the following
short-hand for distribution and quantile functions:
 \begin{align*}
   F_x(\theta) &:= F_{f(x,\rv)}(\theta) = \Prob{f(x,\rv)\leq \theta},\\
   F_x^{-1}(\beta) &:= F_{f(x,\rv)}^{-1}(\beta) = \inf\{\theta\in\rr:\ F_x(\theta) \geq \beta\}.
 \end{align*}
In addition, since the loss function is only defined on $\rvsup$, we frequently take
complements of sets contained in $\rvsup$. Again, to avoid repeated use of cumbersome
notation, the standard notation for complements will apply with respect to $\rvsup$.
That is, for $\riskregion\subseteq\rvsup$ we write $\riskregion^{c}$ in place of $\rvsup \setminus \riskregion$.

Since tail risk measures depend only on those outcomes which are in the $\beta$-tail,
we aim to identify which outcomes
lead to a loss in the $\beta$-tails for a feasible decision. This motivates
the following definition.

\begin{definition}[Risk region]
  For $0 < \beta < 1$ the $\beta$-risk region with respect to the
  decision $x\in\mathcal{X}$ is defined as follows:
  \begin{equation*}
    \riskregion_{x}(\beta) = \{\orv \in \rvsup : F_{x}\left(f(x,\orv)\right) \geq \beta\},
  \end{equation*}
  or equivalently
  \begin{equation}
    \label{eq:tailrisk-def-risk-region}
    \riskregion_{x}(\beta) = \{\orv \in \rvsup : f(x,\orv) \geq F_{x}^{-1}(\beta)\}.
  \end{equation}
  The risk region with respect to the feasible region $\mathcal{X} \subset \rk$ is defined
  to be:
  \begin{equation}
    \label{eq:tailrisk-def-risk-region-2}
    \riskregion_{\mathcal{X}}(\beta) = \bigcup_{x\in\mathcal{X}} \riskregion_x(\beta).
  \end{equation}
  The complement of this region is called the \emph{non-risk region}.
  This can also be written
  \begin{equation}
    \label{eq:tailrisk-non-risk-region}
    \riskregion_{\mathcal{X}}(\beta)^{c} = \bigcap_{x\in\mathcal{X}} \riskregion_x(\beta)^{c}.
  \end{equation}
\end{definition}

The following basic properties of the risk region follow directly from the definition.

\begin{align}
 \text{(i)}\ &0 < \beta' < \beta < 1\ \Rightarrow\ \riskregion_{\mathcal{X}}(\beta) \subseteq \riskregion_{\mathcal{X}}(\beta');\label{eq:tailrisk-beta-increase}\\
 \text{(ii)}\ &\mathcal{X}' \subset \mathcal{X}\ \Rightarrow \riskregion_{\mathcal{X}'}(\beta) \subseteq\riskregion_{\mathcal{X}}(\beta)\label{eq:tailrisk-X-increase} \text{ for all } 0< \beta < 1;\\
 \text{(iii)}\ &\text{If } \orv\mapsto f(x,\orv) \text{ is upper semi-continuous then } \riskregion_{x}(\beta) \text{ is closed and } \riskregion_x(\beta)^c \text{ is open.}
\end{align}

We now state a technical property and prove that this ensures the distribution of the random
vector in a given region completely determines the value of a tail risk
measure. In essence, this condition ensures that there is enough mass in the set
to ensure that the $\beta$-quantile does not depend on the probability distribution outside of it.

\begin{definition}[Aggregation condition]
  Suppose that $\riskregion_{\mathcal{X}}(\beta) \subseteq \riskregion \subset \rvsup$ and that for all $x\in\mathcal{X}$, $\riskregion$ satisfies the following condition:
  \begin{equation}
    \label{eq:agg-condition}
    \Prob{\rv \in \{\orv\in\rvsup: \theta' < f(x,\orv) < F_{x}^{-1}\left(\beta\right)\}\cap\riskregion} > 0 \qquad \forall\ \theta' < F^{-1}_{x}\left(\beta\right).
  \end{equation}
  Then $\riskregion$ is said to satisfy the $\beta$-\emph{aggregation condition}.
\end{definition}

The motivation for the term \emph{aggregation condition} comes from
Theorem~\ref{thr:equiv-rv} which follows. This result ensures that if a set satisfies
the aggregation condition then we can transform the probability
distribution of $\rv$ so that all the mass in the complement of this set
can be aggregated into a single point without affecting the value of the
tail risk measure.  This property is particularly relevant to scenario
generation as if we have such a set, then all scenarios which it does not contain
can be aggregated, reducing the size of the stochastic program.
Note that the $\beta$-aggregation condition does not hold if $\rv$ is a discrete
random vector. However, in this case, the conclusion of the theorem
holds without any extra conditions on $\riskregion$.

\begin{theorem}
  \label{thr:equiv-rv}
  Suppose that $\riskregion_{\mathcal{X}}(\beta) \subseteq \riskregion \subset \rvsup$ and that $\arv$ is a random vector for which
  \begin{equation}
    \label{eq:agg-condition-2}
    \Prob{\rv \in \mathcal{A}} = \Prob{\arv \in \mathcal{A}}\qquad\text{for any measurable } \mathcal{A}\subseteq \riskregion.
  \end{equation}
  Then for any tail risk measure $\rho_{\beta}$ we have $\rho_\beta\left(f(x,\rv)\right) = \rho_\beta\left(f(x,\arv)\right)$ for all $x\in \mathcal{X}$, if one of the following conditions hold:
  \begin{enumerate}[(a)]
  \item $\riskregion$ satisfies the $\beta$-aggregation condition,
  \item $\rv$ is a discrete random vector.
  \end{enumerate}
\end{theorem}

\begin{proof}
  Fix $x\in\mathcal{X}$. To show that $\rho_\beta\left(f(x,\rv)\right) = \rho_\beta\left(f(x,\arv)\right)$ we must show that the $\beta$-quantile
  and the $\beta$-tail distributions of $f(x,\rv)$ and $f(x,\arv)$ are the same.
  Using Lemma~\ref{lem:quantile-beta-tail}, the following two conditions are necessary and sufficient for this to occur:
  \begin{equation*}
    F_{x}(\theta) = F_{f(x,\arv)}(\theta)\ \ \forall\ \theta \geq F_{x}^{-1}\left(\beta\right) \text{  and  }
    F_{f(x,\arv)}(\theta) < \beta\ \ \forall\ \theta < F_{x}^{-1}\left(\beta\right).
  \end{equation*}
  
  In the first case suppose that $\theta' \geq F_{x}^{-1}(\beta)$. Note that as a direct consequence of \eqref{eq:agg-condition-2} we have
  \begin{equation}
    \label{eq:non-risk-region-rv-prob}
    \Prob{\rv\in\mathcal{B}} = \Prob{\arv\in\mathcal{B}} \qquad \text{for any } \mathcal{B}\supseteq\riskregion^{c}.
  \end{equation}
  Now,
  \begin{align*}
    F_{f(x,\arv)}(\theta') &= \Prob{\arv\in\{\orv\in\rvsup:\ f(x,\orv) \leq \theta'\}}\\
    &= \Prob{\arv\in\underbrace{\riskregion^c \cap \{\orv\in\rvsup:\ f(x,\orv) \leq \theta'\}}_{= \riskregion^c}} + \Prob{\arv\in\underbrace{\riskregion \cap \{\orv\in\rvsup:\ f(x,\orv) \leq \theta'\}}_{\subseteq \riskregion}} \\
    & = \Prob{\rv \in \riskregion^c} + \Prob{\rv \in \riskregion \cap \{\orv\in\rvsup: f(x,\orv) \leq \theta'\}}\qquad \text{by \eqref{eq:agg-condition-2} and \eqref{eq:non-risk-region-rv-prob}}\\
    & = \Prob{\rv \in \{\orv\in\rvsup: f(x,\orv) \leq \theta'\}} = F_{x}(\theta') \qquad \text{as required.}
\end{align*}

In the second case we suppose $\theta' < F_{x}^{-1}(\beta)$. We show that $F_{f(x,\arv)}(\theta') < \beta$
  for each of the two conditions (a) and (b) separately.
  In the case where condition (a) holds, that is, when $\riskregion$ satisfies the $\beta$-aggregation condition we have:
  \begin{align*}
    F_{f(x,\arv)} (\theta')  &= \Prob{\arv \in \{\orv\in\rvsup: f(x,\orv) \leq \theta'\}} \leq \Prob{\arv \in \riskregion^c\cup\{\orv\in\rvsup: f(x,\orv) \leq \theta'\}}\\
    &= \Prob{\arv\in\underbrace{\{\orv\in\rvsup:f(x,\orv) < F_{x}^{-1}(\beta)\}}_{\supseteq \riskregion^{c}}} - \Prob{\arv\in \underbrace{\riskregion\cap\{\orv\in\rvsup: \theta' < f(x,\orv) < F_{x}^{-1}(\beta)\}}_{\subseteq\riskregion}}\\
                             &= \Prob{\rv \in \{\orv\in\rvsup: f(x,\orv) < F_{x}^{-1}(\beta)\}} \\
    & \qquad - \Prob{\rv\in\riskregion\cap\{\orv\in\rvsup: \theta' < f(x,\orv) < F_{x}^{-1}(\beta)\}}\qquad\text{by \eqref{eq:agg-condition-2} and \eqref{eq:non-risk-region-rv-prob}}\\
    &< \Prob{\rv\in\{\orv\in\rvsup: f(x,\orv) < F_{x}^{-1}(\beta)\}}\qquad\text{by \eqref{eq:agg-condition}}\\
    &\leq \beta
  \end{align*}
  as required.
  In the case condition (b) holds, that is when $\rv$ is discrete, we have:
  \begin{align*}
    F_{f(x,\arv)}(\theta') &\leq \Prob{f(x,\arv) < F_{x}^{-1}(\beta)}\\
    & = \Prob{f(x, \rv) < F_{x}^{-1}(\beta)} \\ 
    & < \beta \qquad  \text{ since } \rv \text{ is discrete}
  \end{align*}
  as required.
\qed
\end{proof}

It is difficult to verify that a set $\riskregion \supseteq \riskregion_{\mathcal{X}}(\beta)$ satisfies
the $\beta$-aggregation condition by directly checking that the condition \eqref{eq:agg-condition} holds.
The following proposition gives conditions under which it holds immediately for $\riskregion_{\mathcal{X}}(\beta')$
when $\beta' < \beta$.

\begin{proposition}
  Suppose $\beta' < \beta$ and $F_{x}$ is continuous at $F_{x}^{-1}(\beta)$ for all $x\in\mathcal{X}$. 
  Then, $\riskregion_{\mathcal{X}}(\beta')$ satisfies the $\beta$-aggregation condition. That is, for all $x\in\mathcal{X}$
    \begin{equation*}
    \Prob{\rv \in \{\orv\in\rvsup: \theta' < f(x,\orv) < F_{x}^{-1}\left(\beta\right)\}\cap\riskregion_{\mathcal{X}}\left(\beta'\right)} > 0 \qquad \forall\ \theta' < F^{-1}_{x}\left(\beta\right).
    \end{equation*}
\end{proposition}

\begin{proof}
  Fix $x\in \mathcal{X}$. Since $F_{x}$ is continuous at $F_{x}^{-1}(\beta)$ we must have that $F_{x}^{-1}\left(\beta'\right) < F_{x}^{-1}\left(\beta\right)$.
  Now, for all $F_{x}^{-1}\left(\beta'\right) < \theta' < F_{x}^{-1}\left(\beta\right)$, we have $\{\orv\in\rvsup:\theta' < f(x,\orv) < F^{-1}_{x}\left(\beta\right)\} \subset \riskregion_{\mathcal{X}}(\beta')$ and so 
  \begin{equation*}
    \Prob{\rv \in \{\orv\in\rvsup: \theta' < f(x,\orv) < F_{x}^{-1}\left(\beta\right)\}\cap\riskregion_{\mathcal{X}}\left(\beta'\right)} = \Prob{\theta' < f(x,\rv) < F_{x}^{-1}\left(\beta\right)} > 0.
  \end{equation*}
\qed
\end{proof}

For convenience, we now drop $\beta$ from our notation and terminology.
Thus, we refer to the $\beta$-risk region and $\beta$-aggregation condition 
as simply the risk region and aggregation condition respectively,
and write $\riskregion_{\mathcal{X}}(\beta)$ as $\riskregion_{\mathcal{X}}$.

All sets satisfying the aggregation condition must contain the risk
region, however, the aggregation condition does not necessarily hold
for the risk region itself. 

We must impose extra conditions on the
problem to avoid some degenerate cases where the aggregation
condition and the conclusion of Theorem~\ref{thr:equiv-rv} do not hold. 
The following example demonstrates such a degenerate case.

\begin{example}
  Let $\mathcal{X} = \rr^{+}\setminus\{0\}$, $\rvsup=[0,1]$, $\rv\sim \operatorname{Uniform}(0,1)$
  and $f : (x,\orv) \mapsto x\orv$. Then $\riskregion_{x} = [\beta, 1]$ for
  all $x\in\mathcal{X}$, and so $\riskregion_{\mathcal{X}} = [\beta, 1]$. Now, consider the random variable 
  $\phi(\rv)$ where $\phi:\rr\rightarrow\rr$
  is defined as follows:
  $$\phi(\orv) =
  \begin{cases}
    \orv & \text{ if } \orv \geq \beta,\\
    0 & \text{othewise}.
  \end{cases}
  $$
  Since $\phi(\rv) = \rv$ for all $\rv\in\riskregion_{\mathcal{X}}$
  we have $\Prob{\phi(\rv) \in A} = \Prob{\rv\in A}$ for all $A\subseteq\riskregion_{\mathcal{X}}$. On the other hand, we have that $F^{-1}_{f(x,\phi(\rv))}(\beta) = 0 < \beta = F^{-1}_{f(x,\rv)}(\beta)$.
\end{example}

The following result provides extra conditions for
continuous distributions which ensure that the aggregation condition holds
for the risk region $\riskregion_{\mathcal{X}}$.

\begin{proposition}
  \label{prop:agg-condition-suff-conditions}
  Suppose that $\rv$ is a continuous random vector whose support coincides with $\rvsup$, and that the following conditions hold:
  \begin{enumerate}[(i)]
  \item $\orv \mapsto f(x,\orv)$ is continuous for all $x\in\mathcal{X}$,

  \item For each $x\in \mathcal{X}$ there exists $x'\in\mathcal{X}$ such that
    \begin{equation}
      \label{eq:set-intersect}
      \interior{\rvsup}\cap \interior{\riskregion_x\cap\riskregion_{x'}} \neq \emptyset \text{ and } \interior{\rvsup}\cap \interior{\riskregion_{x'}\setminus \riskregion_{x}} \neq \emptyset,
    \end{equation}
  \item $\interior{\rvsup}\cap\interior{\riskregion_{\mathcal{X}}}$ is connected.
  \end{enumerate}
  Then the risk region $\riskregion_{\mathcal{X}}$ satisfies the aggregation condition.
\end{proposition}

\begin{proof}
  Fix $x\in\mathcal{X}$ and $\theta' < F_{x}^{-1}(\beta)$. Pick $x'\in \mathcal{X}$ such that
  \eqref{eq:set-intersect} holds. Also, let ${\orv_0 \in \interior{\rvsup}\cap\interior{\riskregion_{x'}\setminus \riskregion_{x}}}$
  and ${\orv_1\in \interior{\rvsup}\cap \interior{\riskregion_{x}\cap\riskregion_{x'}}}$. Since ${\interior{\rvsup}\cap\interior{\riskregion_{\mathcal{X}}}}$ is connected
  there exists a continuous path from $\orv_{0}$ to $\orv_{1}$. That is, there exists a continuous function
  ${\gamma: [0,1] \rightarrow \interior{\rvsup}\cap\interior{\riskregion_{\mathcal{X}}}}$
  such that $\gamma(0) = \orv_0$ and $\gamma(1) = \orv_1$. Now, ${f(x,\orv_{0}) <  F_{x}^{-1}(\beta)}$ and ${f(x,\orv_1) \geq F_{x}^{-1}(\beta)}$ and
  so given that ${t \mapsto f(x, \gamma(t))}$ is continuous there must exist ${0< t < 1}$ such that
  ${\theta' < f(x,\gamma(t)) < F_{x}^{-1}(\beta)}$. That is,
  $$\interior{\rvsup}\cap \interior{\riskregion_{\mathcal{X}}} \cap \{\orv\in\rvsup: \theta' < f(x,\orv) < F_{x}^{-1}(\beta)\} \neq \emptyset.$$
  This is a non-empty open set contained in the support of $\rv$
  and so has positive probability, hence the aggregation condition holds for $\riskregion_{\mathcal{X}}$.
\qed
\end{proof}

The following proposition gives a condition under which the non-risk region
is convex.  
\begin{proposition}
  \label{prop:tailrisk-convex-non-risk-region}
  Suppose that for each $x\in\mathcal{X}$ the function
  $\orv \mapsto f(x,\orv)$ is convex. Then, the non-risk region $\riskregion_{\mathcal{X}}^{c}$ is convex.
\end{proposition}

\begin{proof}
  For $x\in\mathcal{X}$, if $\orv\mapsto f(x,\orv)$ is convex then the set ${\riskregion_{x}^{c} = \{\orv\in\rvsup: f(x,\orv) < F_{x}^{-1}(\beta)\}}$ must be convex.
  The intersection of convex sets is convex, hence $\riskregion_{\mathcal{X}}^{c} = \bigcap_{x\in\mathcal{X}}\riskregion_{x}^{c}$ is convex.
\qed
\end{proof}

This convexity condition is held by a large class of stochastic programs.
Two-stage stochastic linear programs have loss functions of the following general form:
\begin{equation*}
  Q(x, \rv) = \min_{y} \{\bm{q}^{T}y | \bm{W}y = \bm{h} - \bm{T} x,\ y\geq 0\}
\end{equation*}
where $\bm{q}, y\in \rr^{r}$, $\bm{h}\in\rr^{t}$, $\bm{W}\in\rr^{t\times r}$ and $\bm{T}\in\rr^{t\times k}$,
and $\rv$ is the concatenation of all the stochastic components of the problem; that is,
$\rv^{T} = \left(\bm{q}^{T},\bm{h}^{T},\bm{T}_{1},\ldots,\bm{T}_{t},\bm{W}_{1},\ldots,\bm{W}_{t}\right)$
where $\bm{T}_{i}$ and $\bm{W}_{i}$ denote the $i$-th rows of the matrices $\bm{T}$ and $\bm{W}$
respectively. Standard results in stochastic programming guarantee
that $\orv \mapsto Q(x, \orv)$ is convex if the only random components
of the problem are $\bm{h}$ and $\bm{T}$, that is if $\orv^{T} = (h^{T}, T_{1},\ldots,T_{t})$. See 
for instance \cite[Chapter~3, Theorem~2]{BirgeLouv97}.

The random vector in the following definition plays a special role in our theory.

\begin{definition}[Aggregated random vector]
  \label{def:tailrisk-agg-rv}
  For some set $\riskregion_{\mathcal{X}}\subseteq\riskregion\subset\rvsup$ 
  the \emph{aggregated random vector} is defined as follows:
\begin{equation*}
  \psi_{\riskregion}(\rv) := \begin{cases} \rv &\text{if } \rv \in \riskregion,\\
    \E{\ \rv | \rv\in\riskregion^c\ } & \text{otherwise.} \end{cases}
\end{equation*}
\end{definition}

If $\riskregion$ satisfies the aggregation condition and $\E{\rv|\rv\in\riskregion^{c}}\in\riskregion_{\mathcal{X}}^{c}$ then Theorem \ref{thr:equiv-rv} guarantees that 
${\rho_{\beta}\left(f\left(x,\psi_{\riskregion}(\rv)\right)\right) = \rho_{\beta}\left(f\left(x,\rv\right)\right)}$ for all $x\in\mathcal{X}$. The latter condition holds, for example,
if $\orv \mapsto f(x, \orv)$ is convex for all $x\in\mathcal{X}$, since by Proposition~\ref{prop:tailrisk-convex-non-risk-region} we have that $\riskregion^{c}_{\mathcal{X}}$ is convex and also
$\riskregion^{c}\subseteq\riskregion_{\mathcal{X}}^{c}$.
Under these conditions, as well as preserving the value of the tail risk measure,
the function $\psi_{\riskregion}$ will also preserve the expectation
for affine loss functions.

\begin{corollary}
  \label{cor:explicit-problems}
  Suppose for each $x\in\mathcal{X}$ the function $\orv \mapsto f(x,\orv)$ 
  is affine and for a set $\riskregion\subset \rvsup$ satisfying the aggregation condition
  we have that $\E{\rv|\rv\in \riskregion^{c}} \in \riskregion^{c}$. Then,
  \begin{equation*}
    \rho_{\beta}\left(f\left(x,\psi_{\riskregion}(\rv)\right)\right) = \rho_{\beta}\left(f\left(x,\rv\right)\right) \text{ and }
    \ \E{f\left(x, \psi_{\riskregion}\left(\rv\right)\right)} = \E{f(x,\rv)} \text{ for all } x\in\mathcal{X}.
  \end{equation*}
\end{corollary}

\begin{proof}
  The equality of the tail-risk measures follows immediately from Theorem \ref{thr:equiv-rv}.
  For the expectation function we have
  \begin{align*}
    \E{\psi_{\riskregion}(\rv)} &= \Prob{\rv\in\riskregion}\E{\psi_{\riskregion}(\rv)| \rv\in\riskregion} + \Prob{\rv\in\riskregion^{c}}\E{\psi_{\riskregion}(\rv)|\rv\in\riskregion^{c}}\\
    &=\Prob{\rv\in\riskregion}\E{\rv| \rv\in\riskregion} + \Prob{\rv\in\riskregion^{c}}\E{\rv|\rv\in\riskregion^{c}}=\E{\rv}.
  \end{align*}
  Since $\orv\mapsto f(x,\orv)$ is affine this means that
  \begin{equation*}
    \E{f(x,\psi_{\riskregion}(\rv))} = f(x, \E{\psi_{\riskregion}(\rv)}) = f(x,\E{\rv}) = \E{f(x,\rv)}.
  \end{equation*}
\qed
\end{proof}

\section{Scenario generation}
\label{sec:tailrisk-aggsampling}

In the previous section, we showed that under mild conditions the value of a tail risk measure only
depends on the distribution of outcomes in the risk region. In this section
we demonstrate how this feature may be exploited for
the purposes of scenario generation.

We assume throughout this section that our scenario sets are
constructed from some underlying probabilistic model from which we can
draw independent identically distributed samples. We also assume we have a set $\riskregion_{\mathcal{X}}\subseteq \riskregion\subset\rvsup$ which satisfies
the aggregation condition for the problem under consideration, and for which we can easily test membership. The
set $\riskregion$ may be an \emph{exact risk region}, that is $\riskregion=\riskregion_{\mathcal{X}}$, or it could
a \emph{conservative risk region}, that is $\riskregion \supset \riskregion_{\mathcal{X}}$. To avoid repeating cumbersome
terminology, we simply refer to $\riskregion$ as a risk region, differentiating
between the conservative and exact cases only where necessary. The complement $\riskregion^{c}$ will be referred to as the \emph{aggregation region} for reasons which will become clear. Our general approach to scenario generation is to prioritize the construction of scenarios in the risk region $\mathcal{R}$.

In Section~\ref{sec:tailrisk-aggr-sampl-reduct} we present and analyse a scenario generation method which we call aggregation sampling. In Section~\ref{sec:tailrisk-scengen-alt}
we briefly discuss alternative ways of exploiting risk regions for scenario generation.

\subsection{Aggregation sampling}
\label{sec:tailrisk-aggr-sampl-reduct}

In \emph{aggregation sampling} the user specifies a number of
\emph{risk scenarios}, that is, the number of scenarios to represent the risk region. The algorithm then draws samples from the distribution, storing those samples which lie in the risk region and aggregating those in the aggregation region into a single point. In particular, the samples in the aggregation region are aggregated into their mean. The algorithm terminates when the specified number of risk scenarios has been reached. This is detailed in Algorithm~\ref{alg:agg-sampling}. 
\begin{algorithm}
  \textnormal{}\SetKwData{Left}{left}\SetKwData{This}{this}\SetKwData{Up}{up}
  \SetKwFunction{Union}{Union}\SetKwFunction{FindCompress}{FindCompress}
  \SetKwInOut{Input}{input}\SetKwInOut{Output}{output}
  
  \Input{$\riskregion \subset \rvsup$ set satisfying aggregation condition, $N_{\riskregion}$ number of required risk scenarios}
  \Output{$\{(\orv_s, p_s)\}_{s=1}^{N_{\riskregion} + 1}$ scenario set}
    $n_{\riskregion^c} \leftarrow 0$,\  $n_{\riskregion} \leftarrow 0$,\ $\orv_{\riskregion^{c}}=\mathbf{0}$\;
  \While{$n_{\riskregion} < N_{\riskregion}$} {
    Sample new point $\orv$\;
    \If{$\orv\in\riskregion$}{
      $\orv_{n_{\riskregion}+1} \leftarrow \orv$ ; $n_{\riskregion} \leftarrow n_{\riskregion} + 1 $\;
    }
    \Else{
      $\orv_{\riskregion^{c}} \leftarrow \frac{1}{n_{\riskregion^{c}}+1}\left(n_{\riskregion^{c}}\orv_{\riskregion^{c}} + \orv\right) $;
      $n_{\riskregion^c} \leftarrow n_{\riskregion^c} + 1$
    }
  }
  \If{$n_{\riskregion^c} > 0$}{
    $\orv_{N_{\mathcal{R}}+1} \leftarrow \orv_{\mathcal{\riskregion}^{c}}$\;
  }
  \Else{
    Sample new point $\orv$\;
    $n_{\riskregion^c} \leftarrow 1$;
    $\orv_{N_{\riskregion} + 1} \leftarrow \orv$\;
  }
  \lForEach{$i$ in $1, \ldots, N_{\riskregion}$}{$p_i \leftarrow \frac{1}{n_{\riskregion^c} + N_{\riskregion}}$}
  $p_{N_{\riskregion}+1} \leftarrow \frac{n_{\riskregion^c}}{n_{\riskregion^c} + N_{\riskregion}}$
   \caption{Aggregation sampling}
  \label{alg:agg-sampling}
\end{algorithm}

Aggregation sampling can be thought of as
equivalent to sampling from the aggregated random vector from
Definition~\ref{def:tailrisk-agg-rv} for large sample
sizes. Aggregation sampling is thus consistent with standard Monte Carlo sampling 
only if $\riskregion$ satisfies the aggregation
condition and ${\E{\rv|\rv\in\riskregion^{c}}\in\riskregion^{c}}$.
In Section~\ref{sec:tailrisk-conv-aggr-sampl}, we provide conditions under which we can prove consistency.
Note that it is possible that
the algorithm could terminate without sampling
any scenario in the aggregation region. This
could happen in cases where $\Prob{\rv\in\riskregion^{c}}$
is very small, and
the number of specified risk scenarios $n$ is relatively
small. In this case, to ensure that the algorithm terminates
in a reasonable amount of time and that the scenario
set which the algorithm outputs always has a consistent
number of scenarios, we sample an arbitrary scenario
in place of a scenario representing the aggregated
scenarios. This situation is irrelevant for the
asymptotic analysis of the algorithm.

We now study the performance of our aggregation sampling algorithm. Let
$a=\Prob{\rv\in\mathcal{R}^{c}}$ be the probability of the aggregation region, and $n$ the desired number
of risk scenarios. Let $N(n)$ denote the \emph{effective sample size}
for aggregation sampling, that is, the number of samples drawn until
the algorithm terminates\footnotemark.
\footnotetext{For simplicity of exposition we discount the event that the while loop of the algorithm terminates with $n_{\riskregion^{c}} = 0$ which occurs with probability $(1-a)^{n}$} 
The aggregation sampling algorithm can be viewed as a sequence of
Bernoulli trials where a trial is a success if the corresponding
sample lies in the aggregation region, and which terminates once we have
reached $n$ failures, that is, once we have sampled $n$ scenarios from
the risk region. We can therefore write down the distribution of
$N(n)$:
\begin{equation*}
  N(n) \sim n + \mathcal{NB}(n, a), 
\end{equation*}
where $\mathcal{NB}(n,a)$ denotes a \emph{negative binomial} random variable whose probability mass function
is as follows:
\begin{equation*}
  \binom{k+n-1}{k}(1-a)^{n}a^{k},\qquad k\geq 0.
\end{equation*}
The expected effective sample size of aggregation
sampling is thus:
\begin{equation}
  \label{eq:exp-eff-sample-size}
  \E{N(n)} = n + n \frac{a}{1-a}.
\end{equation}

The expected effective sample size $N(n)$ can be thought of as the
required sample size to construct a scenario set via Monte Carlo
sampling with $n$ scenarios in the risk region $\riskregion$.  Thus,
the greater the expected effective sample size, the greater the
benefit of using aggregation sampling over standard Monte Carlo
sampling. From \eqref{eq:exp-eff-sample-size} we can see that the
expected effective sample size increases as the probability $a$ of the
aggregation region increases. Therefore, when constructing a risk
region $\riskregion\supseteq\riskregion_{\mathcal{X}}$ for the
purposes of scenario generation, it is important that $\mathcal{R}$ is
as tight an approximation of the exact risk region
$\mathcal{R}_{\mathcal{X}}$ as possible in order that $a =
\Prob{\rv\in\riskregion^{c}}$ is as large as possible. Also, the fact that
the advantage of using aggregation sampling over standard Monte Carlo
sample improves as the probability of the risk region increases, also
tells us that this methodology will potentially work better for
problems with higher values of $\beta$ and which are more constrained
due to the relations \eqref{eq:tailrisk-beta-increase} and
\eqref{eq:tailrisk-X-increase}.

\subsection{Alternative approaches}
\label{sec:tailrisk-scengen-alt}

\paragraph{Aggregation reduction}

In aggregation reduction one draws a
fixed number of samples $n$ from the distribution and then aggregates all
those in the aggregation region. As opposed to aggregation sampling, this
method uses a fixed number of samples, but constructs a scenario set with
a random number of scenarios.
Let $R(n)$ denote the number of scenarios which are aggregated
in the aggregation reduction method. Aggregation reduction can similarly be viewed as a sequence of $n$
Bernoulli trials, where success and failure are defined in the same
way as described above. The number of aggregated scenarios in aggregation
reduction is therefore distributed as follows:
\begin{equation*}
  R(n) \sim \mathcal{B}(n, a)
\end{equation*}
where $\mathcal{B}(n,a)$ denotes a binomial random variable and so we have
\begin{equation}
  \label{eq:exp-reduction}
  \E{R(n)} = na.
\end{equation}
Again, the performance of this method, in terms of the expected number
of aggregated scenarios, can be seen to improve as the probability of the aggregation
region increases.

\paragraph{Alternative sampling methods}
The above algorithms and analyses assume that the samples of $\rv$
were identically, independently distributed. However, in principle
the algorithms will work for any unbiased sequence of samples. This
opens up the possibility of enhancing the scenario aggregation and reduction
algorithms by using them in conjunction with variance reduction techniques such as importance sampling, or antithetic sampling \cite{higle1998variance}\footnotemark.
The formulae \eqref{eq:exp-eff-sample-size} and \eqref{eq:exp-reduction}
will still hold, but $a$ will be the probability of a \emph{sample} occuring in
the aggregation region rather than the actual probability of the aggregation region itself.
\footnotetext{Batch sampling methods such as stratified sampling will not
work with aggregation sampling which requires samples to be drawn sequentially.}

\paragraph{Alternative representations of the aggregation region}
The above algorithms can also be generalized in how
they represent the non-risk region. Because aggregation sampling and
aggregating reduction only represent the non-risk region with a single
scenario, they do not in general preserve the overall expectation of
the loss function, or any other statistics of the loss function except for the value of a
tail risk measure.  These algorithms should therefore
generally only be used for problems which only involve tail
risk measures. However, if the loss function is affine (in the sense
of Corollary~\ref{cor:explicit-problems}), then collapsing all points
in the non-risk region to the conditional expectation preserves the
overall expectation.

If expectation or any other statistic of the cost function is used
in the optimization problem then one could represent
the non-risk region region with many scenarios. For example, 
instead of aggregating all scenarios in the non-risk region into a single point
we could apply a clustering algorithm to them such as $k$-means. 
The ideal allocation of points between the risk and non-risk regions
will be problem dependent and is beyond the scope of this paper.

\section{Consistency of aggregation sampling}
\label{sec:tailrisk-conv-aggr-sampl}

The reason that aggregation sampling and aggregation reduction
work is that, for large sample sizes, they are equivalent to sampling
from the aggregated random vector, and if the aggregation
condition holds then the aggregated random vector yields the same
optimization problem as the original random vector. We only prove
consistency for aggregation sampling and not aggregation reduction
as the proofs are very similar. Essentially, the only difference
is that aggregation sampling has the additional complication
of terminating after a random number of samples.

We suppose in this section that we have a sequence
of independently identically distributed (i.i.d.) random vectors $\rv_{1}, \rv_{2}, \ldots$ with the same
distribution as $\rv$, and which are defined on the product probability
space $\Omega^{\infty}$.

\subsection{Uniform convergence of empirical $\beta$-quantiles}
\label{sec:tailrisk-unif-conv-empir}

The i.i.d. sequence of random vectors $\rv_{1}, \rv_{2},\ldots$ can be used to estimate
the distribution and quantile functions of $\rv$. We introduce the additional
short-hand for the empirical distribution and quantile functions:
 \begin{equation*}
  F_{n,x}(\theta) := \frac{1}{n}\sum_{i=1}^n \mathbbm{1}_{\{\orv\in\rvsup: f(x,\orv) \leq \theta\}}(\rv_{i})\ \text{ and }\  F_{n,x}^{-1}(u) := \inf\{\theta\in\rr:\ F_{n,x}(\theta) \geq u\}.
\end{equation*}
Note that these are random-valued functions on the probability space $\Omega^{\infty}$.
It is immediate from the strong law of large numbers that for
all $\bar{x}\in\mathcal{X}$ and $\theta\in\rr$, we have $F_{n,\bar{x}}(\theta) \conwpo F_{\bar{x}}(\theta)$ as $n\to\infty$. In addition, if 
$F_{\bar{x}}$ is strictly increasing at $\theta=F_{\bar{x}}^{-1}(\beta)$, that is for all $\epsilon > 0$
\begin{equation*}
  F_{\bar{x}}\left(F_{\bar{x}}^{-1}(\beta) - \epsilon\right) < \beta < F_{\bar{x}}\left(F_{\bar{x}}^{-1}(\beta) + \epsilon\right).
\end{equation*}
then we also have $F_{n, \bar{x}}^{-1}(\beta)\conwpo F_{\bar{x}}^{-1}(\beta)$ as $n\to\infty$;
see for instance \cite[Chapter 2]{serfling1980approximation}.
The following result extends this pointwise convergence
to a convergence which is uniform with
respect to $x\in\mathcal{X}$.

\begin{theorem}
  \label{thm:tailrisk-conv-unif-quant}
  Suppose the following hold:
  \begin{enumerate}[(i)]
  \item For each $x\in\mathcal{X}$, $F_{x}$ is strictly increasing and continuous at $F_{x}^{-1}(\beta)$,
  \item For all $\bar{x}\in\mathcal{X}$ with probability 1 the mapping $x\mapsto f(x,\rv)$ is continuous at $\bar{x}$,
  \item $\mathcal{X} \subset \rk$ is compact.
   \end{enumerate}
   Then, with probability 1
   \begin{equation*}
     \lim_{n\rightarrow\infty} \sup_{x\in\mathcal{X}} \left|F_{n,x}^{-1}(\beta) - F_{x}^{-1}(\beta)\right| = 0.
   \end{equation*}
\end{theorem}

The proof of this result relies on various continuity
properties of the distribution and quantile functions which are 
provided in Appendix \ref{sec:tailrisk-convergence-results}. Some elements
of the proof below have been adapted from
\cite[Theorem~7.48]{ShapiroEA09}, a result which concerns the uniform
convergence of expectation functions.
\begin{proof}
  Fix $\epsilon_{0} > 0$ and $\bar{x}\in\mathcal{X}$. 
  Since $F_{\bar{x}}$ is right-continuous with left limits, it has only countably many discontinuities, and so there exists $0 <\epsilon < \epsilon_{0}$ such that 
  $F_{\bar{x}}$ is continuous at $F_{\bar{x}}^{-1}(\beta) \pm \epsilon$. Since $F_{\bar{x}}$ is strictly increasing at $F_{\bar{x}}^{-1}(\beta)$, 
  \begin{equation}
    \label{eq:thm2-delta-def}
    \delta := \min \{ \beta - F_{\bar{x}}\left(F_{\bar{x}}^{-1}(\beta) - \epsilon\right),\ F_{\bar{x}}\left(F_{\bar{x}}^{-1}(\beta) + \epsilon\right) - \beta \} > 0.
  \end{equation}
  By Corollary~\ref{cor:tailrisk-prob-cont} in Appendix~\ref{sec:tailrisk-convergence-results} the mapping
  $x\mapsto F_{x}\left(F_{\bar{x}}^{-1}(\beta) - \epsilon\right)$ is continuous at $\bar{x}$. 
  Applying Lemma~\ref{lem:tailrisk-conv-cont} in Appendix~\ref{sec:tailrisk-convergence-results}, there exists a neighborhood $W$ of $\bar{x}$
  such that with probability 1
  \begin{equation*}
    \limsup_{n\rightarrow \infty}\sup_{x\in W\cap\mathcal{X}} \left| F_{n,x}(F_{\bar{x}}^{-1}(\beta) - \epsilon) - F_{n,\bar{x}}(F_{\bar{x}}^{-1}(\beta) - \epsilon)\ \right| < \delta.
  \end{equation*}
  In addition, by the strong law of large numbers, with probability 1
  \begin{equation}
    \lim_{n\rightarrow\infty} \left|F_{n, \bar{x}}\left(F^{-1}_{\bar{x}}(\beta) - \epsilon\right) - F_{\bar{x}}\left(F^{-1}_{\bar{x}}(\beta) - \epsilon\right) \right| = 0.
  \end{equation}
  Note that for all $x\in W\cap\mathcal{X}$
  \begin{align*}
    \left|F_{n, x}\left(F^{-1}_{\bar{x}}(\beta) - \epsilon\right) - F_{\bar{x}}\left(F^{-1}_{\bar{x}}(\beta) - \epsilon\right) \right| & \leq \left| F_{n,x}(F_{\bar{x}}^{-1}(\beta) - \epsilon) - F_{n,\bar{x}}(F_{\bar{x}}^{-1}(\beta) - \epsilon)\ \right| \\
    & \qquad + \left|F_{n, \bar{x}}\left(F^{-1}_{\bar{x}}(\beta) - \epsilon\right) - F_{\bar{x}}\left(F^{-1}_{\bar{x}}(\beta) - \epsilon\right) \right|.
  \end{align*}
  Thus, with probability 1
  \begin{align}
   & \limsup_{n\rightarrow\infty} \sup_{x\in W\cap\mathcal{X}} \left|F_{n, x}\left(F^{-1}_{\bar{x}}(\beta) - \epsilon\right) - F_{\bar{x}}\left(F^{-1}_{\bar{x}}(\beta) - \epsilon\right) \right| \nonumber \\
   \leq & \limsup_{n\rightarrow \infty} \sup_{x\in W\cap\mathcal{X}} \left| F_{n,x}(F_{\bar{x}}^{-1}(\beta) - \epsilon) - F_{n,\bar{x}}(F_{\bar{x}}^{-1}(\beta) - \epsilon)\ \right| \nonumber \\
    & \ \ + \limsup_{n\rightarrow\infty} \left|F_{n, \bar{x}}\left(F^{-1}_{\bar{x}}(\beta) - \epsilon\right) - F_{\bar{x}}\left(F^{-1}_{\bar{x}}(\beta) - \epsilon\right) \right| \nonumber \\
   < &  \delta + 0 = \delta \label{eq:thm-limsup-minus-eps}.
  \end{align}
  Similarly, we can choose $W$ such that with probability 1
  \begin{align}
    \label{eq:thm-limsup-plus-eps}
    \limsup_{n\rightarrow\infty} \sup_{x\in W\cap\mathcal{X}}\left|F_{n, x}\left(F^{-1}_{\bar{x}}(\beta) + \epsilon\right) - F_{\bar{x}}\left(F^{-1}_{\bar{x}}(\beta) + \epsilon\right) \right| < \delta.
  \end{align}
  Using \eqref{eq:thm2-delta-def}, \eqref{eq:thm-limsup-minus-eps} and \eqref{eq:thm-limsup-plus-eps} we can conclude that for all $x\in W\cap\mathcal{X}$ with probability 1
  \begin{equation*}
    \limsup_{n\rightarrow\infty} F_{n, x}\left(F^{-1}_{\bar{x}}(\beta) - \epsilon\right) < \beta < \liminf_{n\rightarrow\infty} F_{n,x}\left(F^{-1}_{\bar{x}}(\beta) + \epsilon\right).
  \end{equation*}
  Hence, we have that for all $x\in W\cap\mathcal{X}$, with probability 1, there exists $N$ such that for all $n > N$
  \begin{equation*}
    F^{-1}_{\bar{x}}(\beta) - \epsilon <  F_{n, x}^{-1}(\beta) \leq F^{-1}_{\bar{x}}(\beta) + \epsilon,
  \end{equation*}
  and so we can conclude that
  \begin{equation}
    \label{eq:tailrisk-sample-quantile-ineq-1}
    \limsup_{n\rightarrow\infty} \sup_{x\in W\cap\mathcal{X}}\left| F_{n,x}^{-1}(\beta) - F_{\bar{x}}^{-1}(\beta) \right| \leq \epsilon < \epsilon_{0}.
  \end{equation}
  Also, by Proposition~\ref{prop:tailrisk-quant-cont} in Appendix~\ref{sec:tailrisk-convergence-results} the function $x\mapsto F^{-1}_{x}(\beta)$
  is continuous and so the neighborhood $W$ can also be chosen so that
  \begin{equation}
    \label{eq:tailrisk-nh-ineq-2}
    \sup_{x\in W\cap \mathcal{X}} \left| F^{-1}_{\bar{x}}(\beta) - F^{-1}_{x}(\beta)\right| < \epsilon_{0},
  \end{equation}
  and so combining \eqref{eq:tailrisk-sample-quantile-ineq-1} and \eqref{eq:tailrisk-nh-ineq-2} we have that with probability 1
  \begin{equation*}
    \limsup_{n\rightarrow\infty} \sup_{x\in W\cap\mathcal{X}} \left| F^{-1}_{n, x}(\beta) -  F^{-1}_{x}(\beta) \right| < 2\epsilon_{0}.
  \end{equation*}
  Finally, since $\mathcal{X}$ is compact, there exists a finite number of points $x_1, \ldots, x_m \in \mathcal{X}$ 
  with corresponding neighborhoods $W_1, \ldots, W_m $ covering $\mathcal{X}$, such that with probability 1, the following holds:
  \begin{equation*}
        \limsup_{n\rightarrow\infty}\sup_{x\in W_j\cap \mathcal{X}}\left| F^{-1}_{n, x}(\beta) - F^{-1}_{x}(\beta)\right| < 2\epsilon_{0} \qquad \text{for } i = 1, \ldots, m
  \end{equation*}
  that is, with probability 1,
  \begin{equation*}
        \limsup_{n\rightarrow\infty} \sup_{x\in\mathcal{X}}\left| F^{-1}_{n, x}(\beta) - F^{-1}_{x}(\beta)\right| < 2\epsilon_{0}.
  \end{equation*}
  Since the choice of $\epsilon_{0}$ was arbitrary the result follows.
\qed
\end{proof}

To facilitate the statement and proofs of the following results we introduce the following index sets
which keep track of the indices of the samples which are in the risk and aggregation regions.
\begin{align*}
  \mathcal{I}_{\riskregion}(n) &= \{1\leq j \leq n:\ \rv_{j}\in\riskregion\},\\
  \mathcal{I}_{\riskregion^{c}}(n) &= \{1\leq j \leq n:\ \rv_{j}\in\riskregion^{c}\}.
\end{align*}
The following corollary shows that we have uniform convergence of the $\beta$-quantiles
when sampling from the aggregated random vector $\psi_{\riskregion}(\rv)$. 
In order to state and prove this result, we introduce the following additional notation for the distribution and quantile functions
for $f(x,\psi_{\riskregion}(\rv))$, and their empirical counterparts for the sample $\psi_{\riskregion}(\rv_{1}), \psi_{\riskregion}(\rv_{2}), \ldots$:
\begin{align*}
  \tilde{F}_{x}(\theta) &= \Prob{f(x, \psi_{\riskregion}(\rv)) \leq \theta}\\
  \tilde{F}^{-1}_{x}(u) &= \inf\{\theta\in\rr:\ \tilde{F}_{x}(\theta) \geq u\}\\
  \tilde{F}_{n,x}(\theta) &= \frac{1}{n}\sum_{i=1}^{n}\mathbbm{1}_{\{\orv\in\rvsup:\ f(x, \orv) \leq \theta\}}\left(\psi_{\riskregion}(\rv_{i})\right)\\
  &=\frac{|\mathcal{I}_{\riskregion^{c}}(n)|}{n}\mathbbm{1}_{\{\orv\in\rvsup:\ f(x, \orv) \leq \theta\}}\left(\E{\rv|\rv\in\riskregion^{c}} \right) + \frac{1}{n}\sum_{i\in\mathcal{I}_{\riskregion}(n)} \mathbbm{1}_{\{\orv\in\rvsup:\ f(x, \orv) \leq \theta\}}(\rv_{i})\\
  \tilde{F}^{-1}_{n,x}(u) &= \inf\{\theta\in\rr: \tilde{F}_{n,x}(\theta)\geq u\}
\end{align*}
Like $F_{n,x}$ and $F_{n,x}^{-1}$, the final two functions are random-valued functions on the probability space $\Omega^{\infty}$.

\begin{corollary}
  \label{cor:agg-rv-unif-conv}
  Let $\riskregion_{\mathcal{X}}\subseteq \riskregion \subset \rd$ be a set
  satisfying the aggregation condition, and suppose that conditions (i)-(iii) from Theorem~\ref{thm:tailrisk-conv-unif-quant}
  hold and in addition:
  \begin{enumerate}[(i)]
    \setcounter{enumi}{3}
  \item $\E{\ \rv | \rv\in\riskregion^c\ } \in \interior{\riskregion_{\mathcal{X}}^c}$.
  \item The mapping $x \mapsto f\left(x, \E{\ \rv | \rv\in\riskregion^c\ }\right)$ is continuous.
  \end{enumerate}
  Then with probability 1
  \begin{equation*}
    \lim_{n\rightarrow\infty} \sup_{x\in\mathcal{X}} |\tilde{F}_{n,x}^{-1}(\beta) - F_{x}^{-1}(\beta)| = 0.
  \end{equation*}
\end{corollary}

\begin{proof}
  Since $\riskregion$ satisfies the aggregation condition, and condition (a) holds,
  by Theorem~\ref{thr:equiv-rv}, we have that $\tilde{F}_{x}^{-1}(\beta) = F_{x}^{-1}(\beta)$
  for all $x\in\mathcal{X}$. Therefore, to prove this result, we will
  apply Theorem~\ref{thm:tailrisk-conv-unif-quant} to 
  $f(x, \psi_{\riskregion}(\rv))$ and so must show that conditions (i)-(iii) from Theorem~\ref{thm:tailrisk-conv-unif-quant} also hold for $f(x, \psi_{\riskregion}(\rv))$. Condition (iii) holds immediately, and
  condition (ii) holds for $f(x, \psi_{\riskregion}(\rv))$ since $x\mapsto f(x,\rv)$
  is continuous with probability 1, and 
  $x \mapsto f\left(x, \E{\ \rv | \rv\in\riskregion^c\ }\right)$ is continuous.

  It remains to show that $\tilde{F}_{x}$ is continuous and strictly increasing at $F_{x}^{-1}(\beta)$ for all $x\in\mathcal{X}$. Fix $x\in\mathcal{X}$.
  Since $F_{x}(\theta)$ and $\tilde{F}_{x}(\theta)$ coincide for
  $\theta \geq F_{x}^{-1}(\beta)$ and $F_{x}$ is strictly increasing at $F_{x}^{-1}(\beta)$, we have
  \begin{align*}
    \tilde{F}_{x}\left(F_{x}^{-1}(\beta) + \epsilon\right) &= F_{x}\left(F_{x}^{-1}(\beta) + \epsilon\right)\\
    & > F_{x}(F_{x}^{-1}(\beta)) \\
    & = \tilde{F}_{x}\left(F_{x}^{-1}(\beta)\right)
  \end{align*}
  and so $\tilde{F}_{x}$ is also strictly increasing at $F_{x}^{-1}(\beta)$. Finally, to show
  that $\tilde{F}_{x}$ is continuous at $F_{x}^{-1}(\beta)$, it suffices to show
  that it is left continuous, since all distribution functions are right continuous.
  For $\epsilon > 0$ sufficiently small we have that $f(x,\E{\ \rv | \rv\in\riskregion^c\ })<F_{x}^{-1}(\beta)$, and so
  \begin{align*}
    \tilde{F}_{x}(F_{x}^{-1}(\beta)) - \tilde{F}_{x}(F_{x}^{-1}(\beta) - \epsilon) &= \Prob{F_{x}^{-1}(\beta) - \epsilon < f(x, \psi_{\riskregion}(\rv)) \leq F_{x}^{-1}(\beta)}\\
    &= \Prob{\rv\in\{\orv\in\rvsup: F_{x}^{-1}(\beta) - \epsilon < f(x, \orv) \leq F_{x}^{-1}(\beta)\}\cap\riskregion}\\
    &\leq \Prob{F_{x}^{-1}(\beta) - \epsilon < f(x, \rv) \leq F_{x}^{-1}(\beta)} \\
    &= F_{x}(F_{x}^{-1}(\beta)) - F_{x}(F_{x}^{-1}(\beta) - \epsilon).
  \end{align*}
  Now, since by assumption $F_{x}$ is continuous at $F_{x}^{-1}(\beta)$, we
  have that $\lim_{\epsilon\downarrow 0}\left(F_{x}(F_{x}^{-1}(\beta)) - F_{x}(F_{x}^{-1}(\beta) - \epsilon\right) = 0$,
  and so must also have $\lim_{\epsilon\downarrow 0}\left(\tilde{F}_{x}(F_{x}^{-1}(\beta)) - \tilde{F}_{x}(F_{x}^{-1}(\beta) - \epsilon\right) = 0$ as required.
\qed
\end{proof}

In the next subsection this result will be used to show
that any point in the interior of the non-risk region $\riskregion^{c}$ will, with probability
1, be in the non-risk region of the sampled scenario set as the sample size grows large.

\subsection{Equivalence of aggregation sampling with sampling from aggregated random vector}
\label{sec:tailrisk-aggregation-sampling}

The main obstacle in showing that aggregation sampling is equivalent
to sampling from the aggregated random vector is to show that the
aggregated scenario in the non-risk region converges almost surely to
the conditional expectation of the non-risk region as the number of
specified risk scenarios tends to infinity.  Recall from Section~\ref{sec:tailrisk-aggsampling} 
that $N(n)$ denotes the effective
sample size in aggregation sampling when we require $n$ risk scenarios
and is distributed as $n+\mathcal{NB}(n, a)$ where $a$ is the
probability of the non-risk region. The purpose of the next lemma is
to show that as $n\to\infty$ the number of samples drawn from
the non-risk region almost surely tends to infinity.

\begin{lemma}
  \label{lem-tailrisk-negbin-infinity}
  Suppose $M(n)\sim\mathcal{NB}(n, p)$ where $0<p<1$. Then with probability 1 we have that ${\lim_{n\rightarrow\infty}M(n) = \infty}$.
\end{lemma}

\begin{proof}
  First note that,
  \begin{equation*}
    \{\lim_{n\rightarrow\infty} M(n) = \infty\}^c = \bigcup_{k\in\mathbb{N}}\left(\bigcap_{n\in\mathbb{N}}\ \bigcup_{t > n}\ \{ M(t) > k\}^c\right)
    = \bigcup_{k\in\mathbb{N}} \limsup_{n\rightarrow\infty}\ \{M(n) \leq k\}.
  \end{equation*}
  Hence, to show that ${\Prob{\{\lim_{n\rightarrow\infty} M(n) = \infty\}} = 1}$ it is
  enough to show for each $k\in \mathbb{N}$ we have that 
  \begin{equation}
    \label{eq:tailrisk-limsup-prob}
      \Prob{\ \limsup_{n\rightarrow\infty}\ \{M(n) \leq k\}\ } = 0.
  \end{equation}
  Now, fix $k\in\mathbb{N}$. Then for all $n\in\mathbb{N}$ we have that
  \begin{equation*}
    \Prob{M(n) = k} = \binom{k+n-1}{k} (1-p)^n\ p^k,
  \end{equation*}
  and in particular,
  \begin{equation*}
    \Prob{M(n+1) = k} = \binom{k + n}{k}(1-p)^{n+1}p^k = \frac{k+n}{n} (1-p)\ \Prob{M(n) = k}.
  \end{equation*}
  For large enough $n$ we have that $\frac{k+n}{n}(1-p) \leq c < 1$ for some constant $c$,
  hence ${\sum_{n=1}^\infty \Prob{M(n) = k} < +\infty}$
  and so
  \begin{equation*}
    \sum_{n=1}^{\infty} \Prob{M(n) \leq k} = \sum_{n=1}^{\infty}\sum_{j=1}^k \Prob{M(n) = j} = \sum_{j=1}^k\sum_{n=1}^{\infty} \Prob{M(n) = j} < \infty.
  \end{equation*}
The result \eqref{eq:tailrisk-limsup-prob} now holds
by the first Borel-Cantelli Lemma \cite[Section~4]{BillingsleyMeasure95}.
\qed
\end{proof}

The next Corollary shows that the strong law of large numbers still applies
for the conditional expectation of the non-risk region in aggregation sampling 
despite the sample size being a random quantity.

\begin{corollary}
  \label{cor:tailrisk-cond-conv}
  Suppose $\E{\norm[\rv]} < +\infty$ and $\Prob{\rv\in\riskregion^{c}} > 0$. Then with probability 1
  \begin{equation*}
    \lim_{n\rightarrow\infty} \norm[\frac{1}{N(n)-n} \sum_{i\in \mathcal{I}_{\riskregion^{c}}(N(n))} \rv_i - \E{\ \rv | \rv\in\riskregion^c\ }] = 0.
  \end{equation*}
\end{corollary}

\begin{proof}
  Define the following measurable subsets of $\Omega^{\infty}$:
  \begin{align*}
    \Omega_{1} &= \{\omega \in \Omega^{\infty} : \lim_{n\to\infty}(N(n)(\omega) - n) = \infty\},\\
    \Omega_{2} &= \{\omega \in \Omega^{\infty} : \lim_{n\to\infty} \frac{1}{n} \sum_{i=1}^{n} \mathbbm{1}_{\riskregion^c}(\rv_{i}(\omega))\rv_{i}(\omega) = \E{\mathbbm{1}_{\riskregion^{c}}(\rv)\rv} \},\\
    \Omega_{3} &= \{\omega \in \Omega^{\infty} : \lim_{n\to\infty} \frac{1}{n} \sum_{i=1}^{n} \mathbbm{1}_{\riskregion^c}(\rv_{i}) = \Prob{\rv\in\riskregion^{c}}\}.
  \end{align*}
  By the strong law of large numbers $\Omega_{2}$ and $\Omega_{3}$ have probability one. Since $N(n) - n \sim \mathcal{NB}(n,a)$, where $a = \Prob{\rv\in\riskregion^{c}}$,
  $\Omega_{1}$ has probability 1 by Lemma \ref{lem-tailrisk-negbin-infinity}. Therefore, $\Omega_{1}\cap\Omega_{2}\cap\Omega_{3}$ has probability
  1 and so it is enough to show that for any $\omega\in\Omega_{1}\cap\Omega_{2}\cap\Omega_{3}$
  we have that 
  \begin{equation*}
        \frac{1}{N(n)(\omega)-n} \sum_{i\in\mathcal{I}_{\riskregion^{c}}\left(N(n)\right)}\rv_i(\omega) \rightarrow \E{\ \rv | \rv\in\riskregion^c\ } \text{ as }  n\rightarrow \infty.
  \end{equation*}
  Let $\omega\in\Omega_{1}\cap\Omega_{2}\cap\Omega_{3}$. Since $\omega\in\Omega_{2}\cap\Omega_{3}$, we have that as $m\to\infty$:
  \begin{equation*}
    \frac{1}{\frac{1}{m} \sum_{i=1}^{m} \mathbbm{1}_{\riskregion^{c}}(\rv_{i}(\omega))} \frac{1}{m}\sum_{i=1}^{m} \mathbbm{1}_{\riskregion^c}(\rv_{i}(\omega))\rv_{i}(\omega) \to \frac{1}{\Prob{\rv\in\riskregion^{c}}}\E{\mathbbm{1}_{\riskregion^{c}}(\rv)\rv} = \E{\rv|\rv\in\riskregion^{c}}.
  \end{equation*}
  Now, fix $\epsilon > 0$. Then there exists $N_{1}(\omega)\in\mathbb{N}$ such
  \begin{equation*}
    m > N_{1}(\omega) \implies \norm[\frac{1}{\frac{1}{m} \sum_{i=1}^{m} \mathbbm{1}_{\riskregion^c}(\rv_{i}(\omega))} \frac{1}{m}\sum_{i=1}^{m} \mathbbm{1}_{\riskregion^c}\left(\rv_{i}(\omega)\right)\rv_{i}(\omega) - \E{\rv|\rv\in\riskregion^{c}}] < \epsilon.
  \end{equation*}
  Since $\omega\in\Omega_{1}$ there exists $N_{2}(\omega)$ such that
  \begin{equation*}
    n > N_{2}(\omega) \implies N(n)(\omega) > N_{1}(\omega).
  \end{equation*}
  Noting that
  \begin{align*}
    \frac{1}{\frac{1}{N(n)(\omega)} \sum_{i=1}^{N(n)(\omega)}  \mathbbm{1}_{\riskregion^c}(\rv_{i}(\omega))} &\frac{1}{N(n)(\omega)}\sum_{i=1}^{N(n)(\omega)} \mathbbm{1}_{\riskregion^c}(\rv_{i}(\omega))\rv_{i}(\omega)\\
    &= \frac{1}{\frac{N(n)(\omega) - n}{N(n)(\omega)}} \frac{1}{N(n)(\omega)}\sum_{i=1}^{N(n)(\omega)} \mathbbm{1}_{\riskregion^c}(\rv_{i}(\omega))\rv_{i}(\omega)\\
    &= \frac{1}{N(n)(\omega)-n}\sum_{i\in\mathcal{I}_{\riskregion^{c}}(N(n))}\rv_{i},
  \end{align*}
  we have that
  \begin{equation*}
    n > N_{2}(\omega) \implies \norm[ \frac{1}{N(n)(\omega)-n}\sum_{i\in\mathcal{I}_{\riskregion^{c}}(N(n))}\rv_{i}(\omega) - \E{\rv|\rv\in\riskregion^{c}}] < \epsilon
  \end{equation*}
  and so $\frac{1}{N(n)(\omega)-n}\sum_{i\in\mathcal{I}_{\riskregion^{c}}(N(n))}\rv_{i}(\omega) \to \E{\rv|\rv\in\riskregion^{c}}$ as $n\rightarrow\infty$.
\qed
\end{proof}

To show that aggregation sampling yields solutions consistent with the underlying
random vector $\rv$, we show that with probability 1, for $n$ large enough, it is
equivalent to sampling from the aggregated random vector $\psi_{\riskregion}(\rv)$, as defined in Definition
\ref{def:tailrisk-agg-rv}. If the region $\riskregion$ satisfies the aggregation condition, and $\conE{\rv}{\rv\in\riskregion^{c}} \in \riskregion_{\mathcal{X}}^{c}$, 
Theorem \ref{thr:equiv-rv} tells us that $\rho_\beta\left(f(x,\psi_{\riskregion}(\rv))\right) = \rho_\beta\left(f(x,\rv)\right)$
for all $x\in\mathcal{X}$. Hence, if sampling is consistent for the risk measure $\rho_{\beta}$, then
aggregation sampling is also consistent.

Noting that $|\mathcal{I}_{\riskregion^{c}}(N(n))| = N(n) - n$, we introduce the following notation for the empirical distribution and quantile functions for loss function
with scenario set constructed by aggregation sampling with $n$ risk scenarios.
\begin{align*}
  \hat{F}_{n,x}(\theta) &= \frac{1}{N(n)}\left((N(n)-n) \mathbbm{1}_{\{\orv\in\rvsup:\ f(x,\orv)\leq\theta\}}\left( \frac{1}{N(n)-n}\sum_{i\in\mathcal{I}_{\riskregion^{c}}(N(n))} \rv_{i} \right)  + \sum_{i\in\mathcal{I}_{\riskregion}(n)} \mathbbm{1}_{\{\orv\in\rvsup:\ f(x,\orv)\leq \theta\}}(\rv_{i})\right)
  \\
  \hat{F}_{n,x}^{-1}(u) &= \inf\{\theta\in\rr:\ \hat{F}_{n,x}(\theta) \geq u\}
\end{align*}
Note that these latter functions will depend on the sample $\rv_{1}, \ldots, \rv_{N(n)}$.

\begin{theorem}
  \label{thm:agg-sample-consistent}
  Let $\riskregion_{\mathcal{X}}\subseteq \riskregion \subset \rd$ be a set
  satisfying the aggregation condition. 
  Suppose that that conditions (i)-(v) from Theorem~\ref{thm:tailrisk-conv-unif-quant} and Corollary~\ref{cor:agg-rv-unif-conv} hold, and in addition that
  \begin{enumerate}[(i)]
  \setcounter{enumi}{5}
  \item For each $x\in\mathcal{X}$, $\xi \mapsto f(x, \xi)$ is continuous at $\E{\ \rv | \rv\in\riskregion^c\ }$
  \end{enumerate}
  Then, with probability 1, for all $u \geq \beta$
  \begin{equation}
    \label{eq:agg-sampling-consistent}
    \lim_{n\rightarrow\infty} \sup_{x\in\mathcal{X}} | \hat{F}_{n,x}^{-1}(u) - \tilde{F}_{n,x}^{-1}(u) | = 0.
  \end{equation}
\end{theorem}

\begin{proof}
  We actually prove a slightly stronger result, that is, with probability 1, there exists $N>0$ such that for all $n>N$, $x\in\mathcal{X}$ and $u\geq\beta$ we have that $\hat{F}^{-1}_{n,x}(u) = \tilde{F}^{-1}_{n,x}(u)$.
  First, note that if
  \begin{equation*}
    \theta \geq \max\left\{ f\left(x, \frac{1}{N(n)-n} \sum_{i\in\mathcal{I}_{\riskregion^{c}}(N(n))} \rv_i\right), f\left(x, \E{\ \rv | \rv\in\riskregion^c\ }\right)\right\}
  \end{equation*}
 then
  \begin{align*}
    \hat{F}_{n, x}(\theta) &= \frac{N(n) - n}{N(n)} \\
                               &\qquad+ \frac{1}{N(n)}\sum_{i\in\mathcal{I}_{\riskregion}(N(n))} \mathbbm{1}_{\{\orv\in\rvsup:\ f(x, \orv) \leq \theta\}}(\rv_{i})\\
    &= \tilde{F}_{N(n), x}(\theta).
  \end{align*}
  So if the following holds with probability 1
  \begin{equation}
    \label{eq:tailrisk-agg-sampling-dist}
    \liminf_{n\rightarrow\infty} \inf_{x\in\mathcal{X}} \left( \tilde{F}_{n,x}^{-1}(\beta) - \max\left( f\left(x, \frac{1}{N(n)-n} \sum_{i\in\mathcal{I}_{\riskregion^{c}}(N(n))} \rv_i\right), f\left(x, \E{\ \rv | \rv\in\riskregion^c\ }\right) \right)\right) > 0
  \end{equation}
  then, by application of Lemma~\ref{lem:quantile-beta-tail}, this implies that with probability 1, there exists $N>0$ such that for all $n>N$ and for all $u \geq \beta$ and $x\in\mathcal{X}$ we have $\hat{F}_{n, x}^{-1}(u) = \tilde{F}_{N(n), x}^{-1}(u)$ as required.
  Since $\E{\ \rv | \rv\in\riskregion^c\ } \in \riskregion_{\mathcal{X}}^{c}$ we have that $f(x, \E{\ \rv | \rv\in\riskregion^c\ }) < F_{x}^{-1}(\beta)$ for all $x\in\mathcal{X}$, and since $\mathcal{X}$ is compact there exists $\delta > 0$ such that
  \begin{equation}
    \label{eq:tailrisk-cond-exp-quantile-dist}
    \inf_{x\in\mathcal{X}}\ \left(F_{x}^{-1}(\beta) - f\left(x, \E{\ \rv | \rv\in\riskregion^c\ }\right)\right) > \delta.
  \end{equation}
  By Corollary~\ref{cor:tailrisk-cond-conv}, the compactness of $\mathcal{X}$ and the continuity of $\xi \mapsto f(x,\xi)$ at $\E{\ \rv | \rv\in\riskregion^c\ }$, we have with probability 1
  \begin{equation}
    \label{eq:tailrisk-conv-cost-cont}
    \limsup_{n\rightarrow\infty} \sup_{x\in\mathcal{X}} \left| f\left(x, \frac{1}{N(n)-n} \sum_{i\in\mathcal{I}_{\riskregion^{c}}(N(n))} \rv_i\right) - f\left(x, \E{\ \rv | \rv\in\riskregion^c}\right)\ \right| = 0.
  \end{equation}
  Also, by Corollary~\ref{cor:agg-rv-unif-conv}, with probability 1
  \begin{equation}
    \label{eq:tailrisk-quant-unif-conv-2}
    \limsup_{n\rightarrow\infty}\sup_{x\in\mathcal{X}}\left| F_{x}^{-1}\left(\beta\right) - \tilde{F}_{N(n), x}^{-1}\left(\beta\right)\right| = 0.
\end{equation}
Thus, letting $z(x) = F_{x}^{-1}\left(\beta\right) - f\left(x, \E{\ \rv | \rv\in\riskregion^c}\right)$, we also have with probability 1 that
\begin{equation*}
\lim_{n\rightarrow\infty}\sup_{x\in\mathcal{X}}\left| \left(\tilde{F}_{N(n), x}^{-1}(\beta) - f\left(x, \frac{1}{N(n)-n} \sum_{i\in\mathcal{I}_{\riskregion^{c}}(N(n))} \rv_i\right)\right) - z(x)\right| = 0.
\end{equation*}
In particular, with probability 1 there exists $N$ such that for $n > N$
\begin{equation}
  \label{eq:tail-risk-thm3-conv-delta}
\sup_{x\in\mathcal{X}}\left| \left(\tilde{F}_{N(n), x}^{-1}(\beta) - f\left(x, \frac{1}{N(n)-n} \sum_{i\in\mathcal{I}_{\riskregion^{c}}(N(n))} \rv_i\right)\right) - z(x)\right| < \frac{\delta}{2}.
\end{equation}
In which case, for $n > N$
\begin{align*}
  &\inf_{x\in\mathcal{X}} \left( \tilde{F}_{N(n), x}^{-1}(\beta) - f\left(x, \frac{1}{N(n)-n} \sum_{i\in\mathcal{I}_{\riskregion^{c}}(N(n))} \rv_i\right) \right) \\
 = &\inf_{x\in\mathcal{X}} \left( z(x) + \tilde{F}_{N(n), x}^{-1}(\beta) - f\left(x, \frac{1}{N(n)-n} \sum_{i\in\mathcal{I}_{\riskregion^{c}}(N(n))} \rv_i\right)  - z(x)\right) \\
  \geq &\inf_{x\in\mathcal{X}} z(x) - \sup_{x\in\mathcal{X}}\left| \left(\tilde{F}_{N(n), x}^{-1}(\beta) - f\left(x, \frac{1}{N(n)-n} \sum_{i\in\mathcal{I}_{\riskregion^{c}}(N(n))} \rv_i\right)\right) - z(x)\right| \\
  > & \delta - \frac{\delta}{2} = \frac{\delta}{2} \qquad \text{by \eqref{eq:tailrisk-cond-exp-quantile-dist} and \eqref{eq:tail-risk-thm3-conv-delta}.}
\end{align*}

We can similarly show that  $\limsup_{n\rightarrow\infty} \inf_{x\in\mathcal{X}} \left( \tilde{F}_{N(n), x}^{-1}(\beta) -  f\left(x, \E{\ \rv | \rv\in\riskregion^c}\right)\right) > 0$ holds with probability 1. Hence \eqref{eq:tailrisk-agg-sampling-dist} holds with probability 1 and the proof is complete.
\qed
\end{proof}

Note that although the continuity conditions (ii), (v) and (vi) look complicated, the loss function $f : \mathcal{X}\times \Xi \rightarrow \rr$ will typically be continuous everywhere, and so these will be satisfied automatically.

\section{A conservative risk region for monotonic loss functions}
\label{sec:monotonic}

In order to use risk regions for scenario generation, we need to have a
characterization of the risk region which conveniently allows us to test membership.
In general this is a difficult as the risk region depends on the loss function,
the distribution and the problem constraints. Therefore, as a proof-of-concept, in the following two sections we derive risk regions for two classes of problems.
In this section we propose a conservative risk region for problems
which have monotonic loss functions.

\begin{definition}[Monotonic loss function]
  A loss function $f: \mathcal{X}\times\rvsup \rightarrow \rr$ is monotonic increasing
  if for all $x\in\mathcal{X}$ and $\orv, \tilde{\orv}\in \rvsup$ such that $\orv < \tilde{\orv}$ we have $f(x, \orv) < f(x, \tilde{\orv})$. Similarly, we say it is monotonic decreasing if for all $x\in\mathcal{X}$ and $\orv, \tilde{\orv}\in \rvsup$ such that $\orv<\tilde{\orv}$ we have $f(x,\orv) > f(x, \tilde{\orv})$.
\end{definition}

Monotonic loss functions occur naturally in stochastic linear programming.
The following result presents a class of loss functions which arise in the
context of network design, and gives conditions under
which they are monotonic.

\begin{proposition}
  \label{prop:mono-recourse}
  Suppose $\mathcal{X}\subseteq \rk_{+}$, $\rvsup\subseteq\rd_{+}$ and the loss function $Q(x, \xi)$
  is defined to be the optimal value to the following linear program:
  \begin{align}
    \min_{y,z}\ & q^{T} y + u^{T} z \label{eq:net-des-obj} \\ 
    \text{such that } & Wy + z \geq \orv \label{eq:net-des-dem-sat}\\
    & By \leq b \label{eq:net-des-supply}\\
    & Ty \leq Vx \label{eq:net-des-open-cons}\\
    & \mathcal{N}y = 0 \label{eq:net-des-net-flow}\\
    & y, z \geq 0 \label{eq:net-des-pos},
  \end{align}
  where $W, B, T, V, \mathcal{N}$ are matrices and $q, u, b$ are vectors of compatible dimensions.
  Then, $Q(x, \orv)$ is monotonic increasing under the following conditions:
  \begin{enumerate}
  \item $q, u > 0$,
  \item $b\geq 0$,
  \item $W, B, T, V \geq 0$.
  \end{enumerate}
\end{proposition}

\begin{proof}
  Fix $x\in\mathcal{X}$. The problem is always feasible since $y=0$ and $z=\orv$ is a feasible solution. Since $x\geq 0$ and $q, u > 0$ the problem \eqref{eq:net-des-obj}--\eqref{eq:net-des-pos} is bounded below by zero. In addition, when $\orv \geq 0$ with at least one component strictly greater than zero, the optimal
  solution $(y^{*}, z^{*})$ must contain at least one strictly positive element due to constraint \eqref{eq:net-des-dem-sat} and the fact that $W \geq 0$, and so in this case the optimal value is 
  strictly positive. Because the problem is both bounded below and feasible,
  strong duality applies and so $Q(x,\orv)$ is also equal to the optimal
  solution to the dual problem:
  \begin{align}
    \max_{\pi,\nu,\eta,\lambda} &\ \orv^{T}\pi - x^{T}V^{T} \nu - b^{T}\eta \label{eq:net-dual-obs} \\
    \text{such that } & W^{T}\pi - T^{T}\nu - B^{T} \eta + \mathcal{N}^{T} \lambda \leq q \label{eq:net-dual-blocks} \\
    & \pi \leq u \label{eq:net-dual-pi}\\
    & \pi, \nu, \eta \geq 0 \label{eq:net-dual-pos}.
  \end{align}
  Let $\bar{\orv}, \tilde{\orv} \in \rvsup$ be such that $\bar{\orv} < \tilde{\orv}$. In the first case suppose that $\bar{\orv}\neq 0$,
  and let $(\bar{\pi},\bar{\nu},\bar{\eta},\bar{\lambda})$ be the optimal dual
  variables for \eqref{eq:net-dual-obs}--\eqref{eq:net-dual-pos} for $\orv = \bar{\orv}$.
  As discussed above, this means that $Q(x, \bar{\orv}) > 0$, and given that $x^{T}V^{T}\bar{\nu} + b^{T}\bar{\eta} \geq 0$, at least one component
  of $\bar{\pi}$ will be greater than zero in order for the objective of the dual to be strictly positive. Now, $(\bar{\pi},\bar{\nu},\bar{\eta},\bar{\lambda})$ is also a feasible
  solution to the dual problem with $\orv=\tilde{\orv}$ and so
  \begin{align*}
    Q(x,\bar{\orv}) &= \bar{\orv}^{T}\bar{\pi} - x^{T}V^{T}\bar{\nu} - b^{T}\bar{\eta}\\
               &< \tilde{\orv}^{T}\bar{\pi} - x^{T}V^{T}\bar{\nu} - b^{T}\bar{\eta}\\
               &\leq Q(x,\tilde{\orv}).
  \end{align*}
  In the second case suppose that $\bar{\orv} = 0$. In this case $y=0,\ z=0$ is a feasible solution to the primal problem
  \eqref{eq:net-des-obj}--\eqref{eq:net-des-pos} with $\orv=\bar{\orv}$ and this solution has an objective value of zero. Since the objective is bounded
  below by zero, this means this solution is also optimal and so $Q(x,\bar{\orv}) = 0$. Since $\tilde{\orv} > 0$
  we have that $Q(x,\tilde{\orv}) > 0$, and so $Q(x,\bar{\orv}) < Q(x,\tilde{\orv})$. Hence $Q(x,\orv)$ is monotonic as required.
\qed
\end{proof}

This recourse function arises in stochastic network design, and the problem formulation
in the previous proposition was adapted from a model in the paper \cite{SantosoEtAl2005}.
In this type of problem, we have a network consisting of suppliers, processing units, and customers,
and decisions must be made relating to opening facilities and the capacities
of nodes and arcs. The problem which defines the recourse function $Q(x, \orv)$ depends on the capacity and opening decisions $x$
of the first stage, and the demand of the customers $\orv$. The aim of the problem is construct of flow of products $y$
which minimize transportation costs for satisfying customers demand, plus penalties for any unsatisfied demand $z$.

For a problem with a monotonic loss function, the following result defines a conservative risk region.

\begin{theorem}
  Suppose the loss function $f: \mathcal{X}\times\rvsup \rightarrow \rr$ is monotonic increasing. Then the following
  set is a conservative risk region:
  \begin{equation}
    \label{eq:monotonic-risk-region}
    \riskregion_{1} = \{ \orv\in\rvsup : \Prob{\rv > \orv} \leq 1 - \beta\}.
  \end{equation}
  Similarly, if the loss function is monotonic decreasing then the following set is
  a conservative risk region:
  \begin{equation}
    \label{eq:monotonic-risk-region-dec}
    \riskregion_{2} = \{ \orv\in\rvsup : \Prob{\rv < \orv} \leq 1 - \beta\}.
  \end{equation}
\end{theorem}

\begin{proof}
  Suppose $f(x,\orv)$ is monotonic increasing and let $\orv\in\riskregion_{\mathcal{X}}$, then
  \begin{align*}
    \Prob{\rv > \orv} & \leq \Prob{f(x,\rv) > f(x,\orv)} \qquad \text{by monotonicity}\\
    & = 1 - \underbrace{\Prob{f(x,\rv) \leq f(x,\orv)}}_{\geq \beta}\\
    & \leq 1 - \beta
  \end{align*}
  and so $\orv \in \riskregion_{1}$ as required. The set $\riskregion_{2}$ can similarly be shown to be
  a conservative risk region when $f(x, \orv)$ is monotonic decreasing.
\qed
\end{proof}

\section{An exact risk region for the portfolio selection problem}
\label{sec:tailrisk-portfolio}

In this section, we characterize exactly the risk region of the
portfolio selection problem when the distribution of asset returns belongs to a certain
class of distributions.

In the portfolio selection problem one aims to choose a portfolio
of financial assets with uncertain returns. For $i = 1,\ldots, d$, let $x_{i}$ denote
the amount to invest in asset $i$, and $\rv_{i}$ the random
return of asset $i$. The loss function in this problem is the 
negative total return, that is $f(x, \rv) = \sum_{i=1}^d -x_i \rv_i = -x^T \rv$,
and $\rvsup = \rd$.
The set $\mathcal{X}\subset\rd$ of feasible portfolios may encompass
constraints like no short-selling ($x \geq 0$), total investment
($\sum_{i=1}^{d} x_{i} = 1$) and quotas on certain stocks
($x \leq c$).

The following corollary gives sufficient conditions for the risk region
to satisfy the aggregation condition, and for aggregation sampling
to be consistent.

\begin{corollary}
  \label{cor:tailrisk-agg-condition-portfolio}
  Suppose that $\riskregion \supseteq \riskregion_{\mathcal{X}}$ and
  that the following conditions hold:
  \begin{enumerate}
  \item $\rv$ is continuous with support $\rd$,
  \item There exists $x_{1},x_{2}\in\mathcal{X}$ which are linearly independent,
  \item $0\notin\mathcal{X}$,
  \item $\mathcal{X}$ is compact.
  \end{enumerate}
  Then $\riskregion$ satisfies the aggregation condition, and
  aggregation sampling with respect to $\riskregion$ is consistent in the sense of Theorem \ref{thm:agg-sample-consistent}.
\end{corollary}

\begin{proof}
  To prove that $\riskregion$ satisfies the aggregation condition, it is enough to show that $\riskregion_{\mathcal{X}}$
  satisfies the aggregation condition. We prove this by showing that all the conditions
  of Proposition \ref{prop:agg-condition-suff-conditions}
  hold. Note that $x\mapsto -x^{T}\orv$ is continuous
  so condition (i) of Proposition~\ref{prop:agg-condition-suff-conditions} holds immediately.

  For each $x\in\mathcal{X}$ the interior of the corresponding risk region
  and non-risk region are open half-spaces:
  \begin{equation*}
    \interior{\riskregion_{x}} = \{ \orv\in\rd: -x^{T}\orv > F_{x}^{-1}(\beta)\}\ \text{ and }\ \interior{\riskregion_{x}^{c}} = \{ \orv\in\rd: -x^{T}\orv < F_{x}^{-1}(\beta)\}.
  \end{equation*}
  Fix $\bar{x}\in\mathcal{X}$. Then either $\bar{x}$ is linearly independent to
  $x_{1}$ or it is linearly independent to $x_{2}$. Assume
  it is linearly independent to $x_{1}$. Now, $\interior{\riskregion_{\bar{x}}}$
  and $\interior{\riskregion_{x_{1}}}$ are non-parallel half-spaces
  and so both $\interior{\riskregion_{\bar{x}}\cap\riskregion_{x_{1}}}$ and $\interior{\riskregion_{x_{1}}\setminus \riskregion_{\bar{x}}} = \interior{\riskregion_{x_{1}}} \cap \interior{\riskregion_{\bar{x}}^{c}}$
  are non-empty, and since we also have $\rvsup=\rd$, condition (ii) of 
  Proposition~\ref{prop:agg-condition-suff-conditions} is satisfied. 

  Since $\riskregion_{x_{1}}$ and $\riskregion_{x_{2}}$
  are non-parallel half-spaces, their union $\riskregion_{x_{1}}\cup\riskregion_{x_{2}}$
  is connected. Similarly, for any $x\in\mathcal{X}$, we must have $\riskregion_{x}$
  being non-parallel with either $\riskregion_{x_{1}}$ or $\riskregion_{x_{2}}$
  and so $\riskregion_{x}\cup\riskregion_{x_{1}}\cup\riskregion_{x_{2}}$
  must also be connected. Hence, $\riskregion_{\mathcal{X}} = \bigcup_{x\in\mathcal{X}}\left(\riskregion_{x}\cup\riskregion_{x_{1}}\cup\riskregion_{x_{2}}\right)$ is connected
  so condition (iii) of Proposition~\ref{prop:agg-condition-suff-conditions} is also satisfied.
  Hence $\riskregion$ satisfies the aggregation condition.

  We show that aggregation sampling is consistent in the sense of Theorem \ref{thm:agg-sample-consistent}
  by showing that the conditions of this theorem hold. 
  We have already shown that condition (i) of Theorem~\ref{thm:agg-sample-consistent} holds.
  The loss function is continuous, and so condition (ii) of Theorem~\ref{thm:agg-sample-consistent} holds.
  Let $\epsilon > 0$, then
  \begin{equation*}
   F_{x}\left(F_{x}^{-1}(\beta) + \epsilon\right) -  F_{x}\left(F_{x}^{-1}(\beta)\right) = 
    \Prob{\rv \in\{\orv\in\rd: F_{x}^{-1}(\beta) < -x^{T}\orv \leq F_{x}^{-1}(\beta)+\epsilon\}}.
  \end{equation*}
  Since $x\neq 0$, the set defining this event has a non-empty interior, and since the support of $\rv$ is $\rd$, this probability is greater than zero.
  Hence, $F_{x}$ is increasing at $F_{x}^{-1}(\beta)$. Since $\rv$ is continuous, we also have that $F_{x}$ is continuous and so condition (iii) of 
  Theorem~\ref{thm:agg-sample-consistent} holds.

  By Proposition~\ref{prop:tailrisk-convex-non-risk-region} $\riskregion_{\mathcal{X}}^{c}$ is convex,
  and since $\riskregion^{c}\subseteq\riskregion_{\mathcal{X}}^{c}$ and $\riskregion_{\mathcal{X}}$ is open we have $\E{\ \rv | \rv\in\riskregion^c\ } \in \interior{\riskregion_{\mathcal{X}}^c}$,
  and so condition (iv) of Theorem~\ref{thm:agg-sample-consistent} holds.
  Finally, condition (v) of Theorem~\ref{thm:agg-sample-consistent} holds by assumption
  and so aggregation sampling with the set $\riskregion$ is consistent in
  sense of Theorem~\ref{thm:agg-sample-consistent}.
\qed
\end{proof}

Elliptical distributions are a general class of distributions which
include among others the multivariate Normal and multivariate
$t$-distributions. See \cite{FangKotzNg198911} for a full overview of
the subject.

\begin{definition}[Spherical and Elliptical Distributions]
  \label{def:elliptical}
  Let $\bm{\zeta}$ be a random vector in $\rd$. Then $\bm{\zeta}$ is said to be \emph{spherical}
if its distribution is invariant under orthonormal transformations; that is, if
$$\bm{\zeta} \sim U\bm{\zeta} \qquad \text{for all } U\in\rr^{d\times d}\text{ orthonormal}.$$
  Let $\rv$ be a random vector in $\rd$. Then $\rv$ is said to be \emph{elliptical}
if it can be written $\rv = P\bm{\zeta} + \mu$ where $P\in \rr^{d\times d}$ is non-singular,
$\mu\in\rd$, and $\bm{\zeta}$ is random vector with spherical distribution. 
We will denote this $\rv\sim\mathrm{Elliptical}(\bm{\zeta}, P, \mu)$.
\end{definition}

An important property of elliptical distributions is that
for any $x\in\rd$ we can characterize exactly the distribution of $x^{T}\rv$.
If $\rv\sim\mathrm{Elliptical}(\bm{\zeta}, P, \mu)$ then:
\begin{equation}
  \label{eq:elliptical-return-dist}
  -x^T\rv \sim \norm[Px] \bm{\zeta}_1 - x^T \mu,
\end{equation}
where $\bm{\zeta}_{1}$ is the first component of the random vector $\bm{\zeta}$, and
$\norm$ denotes the standard Euclidean norm. 
By \eqref{eq:elliptical-return-dist} the $\beta$-quantile of the loss of a
portfolio is as follows:
\begin{equation*}
  \label{eq:elliptical-return-var}
  F_{x}^{-1}(\beta) = \norm[Px] F^{-1}_{\bm{\zeta}_1}(\beta) - x^T \mu.
\end{equation*}
Therefore, the exact risk region for $\rv\sim\mathrm{Elliptical}(\bm{\zeta},P,\mu)$, is as follows:
\begin{equation}
  \label{eq:reg-agg-no-constraints}
  \bigcup_{x\in\mathcal{X}}\{\orv\in\rd: -x^T \orv \geq \norm[Px] F_{\bm{\zeta}_1}^{-1}(\beta) - x^T\mu\}.
\end{equation}
This characterization is not practical for testing whether or not a point belongs
to the risk region, which is required for our scenario generation algorithms.
However, a more convenient form is available in the case where $\mathcal{X}\subset\rd$ is convex. Before stating the result, we recall the concept of a projection onto a convex set.

\begin{definition}[Projection]
  \label{def:proj}
  Let $C \subset \rd$ be a closed convex set. Then for
  any point $\orv\in\rd$, we define the projection of $\orv$ onto $C$
  to be the unique point $p_C(\orv)\in C$ such that
  $\inf_{x\in C} \norm[x-\orv] = \norm[p_C(\orv) - \orv]$.
\end{definition}

By a slight abuse of notation, for a set $\mathcal{A}\subset\rd$ and a matrix $T\in\rr^{d\times d}$, we write ${T\left(\mathcal{A}\right) := \{ T\orv : \orv\in\mathcal{A}\}}$. Finally, recall that
the conic hull of a set $\mathcal{A}\subset\rd$, which we denote $\conic{\mathcal{A}}$, is the smallest convex cone containing $\mathcal{A}$.

\begin{theorem}
  Suppose  $\rv\sim\mathrm{Elliptical}(\bm{\zeta}, P, \mu)$ and $\mathcal{X}\subset\rd$ is a convex set. Then the exact non-risk region in \eqref{eq:reg-agg-no-constraints} can be written as follows:
\begin{equation}
  \label{eq:agg-region-cone-constraint}
  \riskregion_{\mathcal{X}}^{c} = P^{T}\left(\{\tilde{\orv}\in\rd: \norm[p_{K'}(\tilde{\orv} - \tilde{\mu})] < F_{\bm{\zeta}_1}^{-1}\left(\beta\right)\}\right)
\end{equation}
where $\tilde{\mu}=(P^{T})^{-1}\mu$, $K' = -PK$ and $K = \conic{\mathcal{X}}$.
\end{theorem}

\begin{proof}
  \begin{align*}
    \riskregion_{\mathcal{X}}^{c} &= \{\orv\in\rd : -x^{T}\orv < \norm[Px] F_{\bm{\zeta}_{1}}^{-1}(\beta) - x^{T}\mu\ \forall x\in\mathcal{X}\}\\
    &= \{\orv\in\rd : \tilde{x}^{T}\orv < \norm[P\tilde{x}] F_{\bm{\zeta}_{1}}^{-1}(\beta) + \tilde{x}^{T}\mu\ \forall \tilde{x}\in -\mathcal{X}\}\\
    &= \{\orv\in\rd : \tilde{x}^{T}(\orv-\mu) < \norm[P\tilde{x}] F_{\bm{\zeta}_{1}}^{-1}(\beta)\ \forall \tilde{x}\in -\mathcal{X}\}\\
    &= P^{T}\left(\{\tilde{\orv}\in\rd : \norm[p_{K'}(\tilde{\orv} - \tilde{\mu})] < F_{\bm{\zeta}_{1}}^{-1}(\beta)\}\right)\ \text{ by application of Corollary~\ref{cor:cone-proj-ellipse-2} in Appendix~\ref{sec:tailrisk-convex}}\\
    \end{align*}
\qed
\end{proof}

\section{Numerical tests}
\label{sec:tailrisk-numtests}

In this section, we test the performance of the methodology developed in this paper.
For the portfolio selection problem, when $\mathcal{X}\subseteq\rr_{+}^{d}\setminus\{0\}$ the loss function $f(x,\orv) = -x^{T}\orv$
is monotonic decreasing. We therefore use this problem throughout this section to
test both the conservative risk region presented in Section~\ref{sec:monotonic},
and the exact risk region presented in Section~\ref{sec:tailrisk-portfolio}.

In order to test whether a point belongs to the exact non-risk region in \eqref{eq:agg-region-cone-constraint}
requires the projection of a point onto a convex cone. This can be done by solving
a small linear complementarity problem. See \cite{Ujvari2007} or our follow-up
paper \cite{Fairbrother15b} for more details. We solve linear
complementarity problems using code from the Siconos numerics library \cite{A-AcaPer07}.
To test whether a point $\orv\in\rvsup$ belongs to the conservative risk region in \eqref{eq:monotonic-risk-region-dec},
involves the evaluation of the probability $\Prob{\rv < \orv}$. Since calculating this probability
exactly involves evaluating a multidimensional integral we approximate the probability
by taking a large sample from $\rv$, and using the empirical distribution function of this sample.
Repeatedly testing membership of both types of risk region is therefore computationally intensive.
Ways of mitigating this issue are discussed in our follow-up paper \cite{Fairbrother15b}.
These membership tests, and the aggregation sampling algorithm have been implemented and made available as a package for the Julia programming language \cite{fairbrother17package}. All experiments were conducted on a laptop with an Intel
Core i7-720QM CPU at 1.6 GHz.

\subsection{Probability of risk regions}
\label{sec:prob-risk-regi}

As discussed in Section~\ref{sec:tailrisk-aggr-sampl-reduct}, the performance
of the aggregation sampling algorithm with respect to standard
Monte Carlo sampling improves as the probability of the aggregation
region increases. In this first experiment we observe the behavior
of this probability over a range of dimensions.

For this experiment, we suppose that $K = \conic{\mathcal{X}} = \rd_{+}$,
and that the random vector follows a multivariate Normal distribution $\mathcal{N}(0, \Lambda(\rho))$,
where the covariance matrix $\Lambda(\rho)$, for $0\leq\rho<1$, is defined as follows:
\begin{equation*}
  \Lambda_{ij}(\rho) =
  \begin{cases}
    \rho & \text{ if } i\neq j,\\
    1 & \text{ otherwise.}
  \end{cases}
\end{equation*}
In particular, we calculate the probability for the case $\rho = 0$, that is
where the asset returns are independently distributed, and the case $\rho = 0.3$, that is
where the asset returns are positively correlated. The probabilities
of the non-risk regions are estimated by sampling and testing membership
for 20000 points.

The results of this experiment are plotted in Figure~\ref{fig:tailrisk-probs-riskregion}.
In Figures~\ref{fig:tailrisk-monotonic-prob} and \ref{fig:tailrisk-nonrisk-prob-ps}
are plotted the probabilities of the conservative and exact aggregation regions.
To aid the readers' intuition we have also plotted a reduced scenario
set in two dimensions using conservative and exact risk regions in Figures~\ref{fig:tailrisk-monotonic-scen-set}
and \ref{fig:tailrisk-nonrisk-scenset-ps} for $\rho=0.3$ and $\beta=0.95$.

The figures show that not only is the probability of the conservative aggregation region smaller
than that of the exact aggregation region but also it decays much more quickly.
This emphasizes the importance of using an exact risk region for aggregation
sampling if possible. Interestingly the probability of the aggregation regions
for the correlated asset returns is greater and decays more slows than that of
the independent asset returns. This tells us that, in addition to the loss function, 
the performance of our methodology depends strongly on the distribution of the random
vector. Although the probability of the conservative aggregation region decays fairly
rapidly, it remains non-negligible for random vectors of a moderate dimension, around 15,
for the correlated asset returns. For exact aggregation regions, the probability
remains high for the correlated asset returns for up to a dimension of 40.

\begin{figure}[h]
  \centering
  \begin{subfigure}[b]{0.45\textwidth}
    \includegraphics[width=\textwidth]{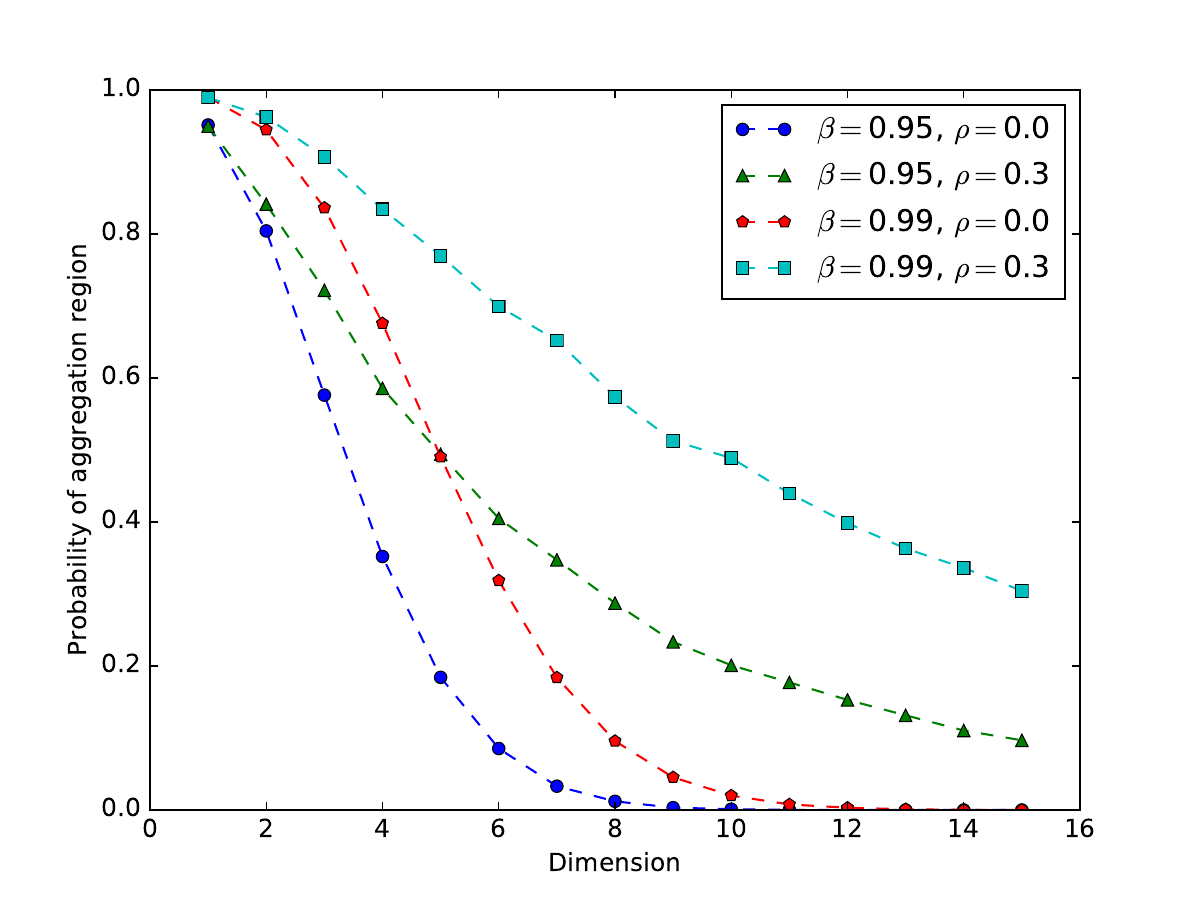}
    \caption{Probability of conservative aggregation region}
    \label{fig:tailrisk-monotonic-prob}
  \end{subfigure}
  \begin{subfigure}[b]{0.45\textwidth}
    \includegraphics[width=\textwidth]{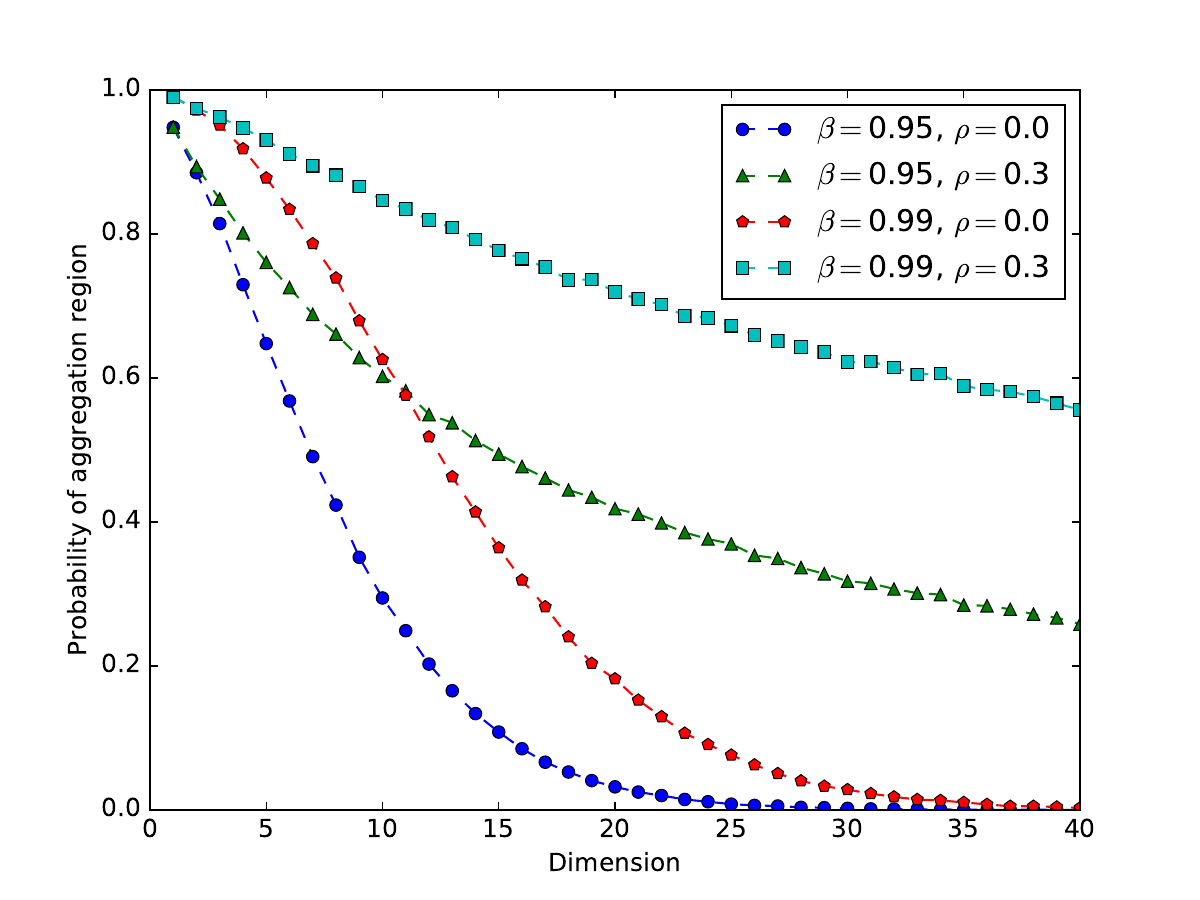}
    \caption{Probability of exact risk region}
    \label{fig:tailrisk-nonrisk-prob-ps}
  \end{subfigure}
  \begin{subfigure}[b]{0.45\textwidth}
    \includegraphics[width=\textwidth]{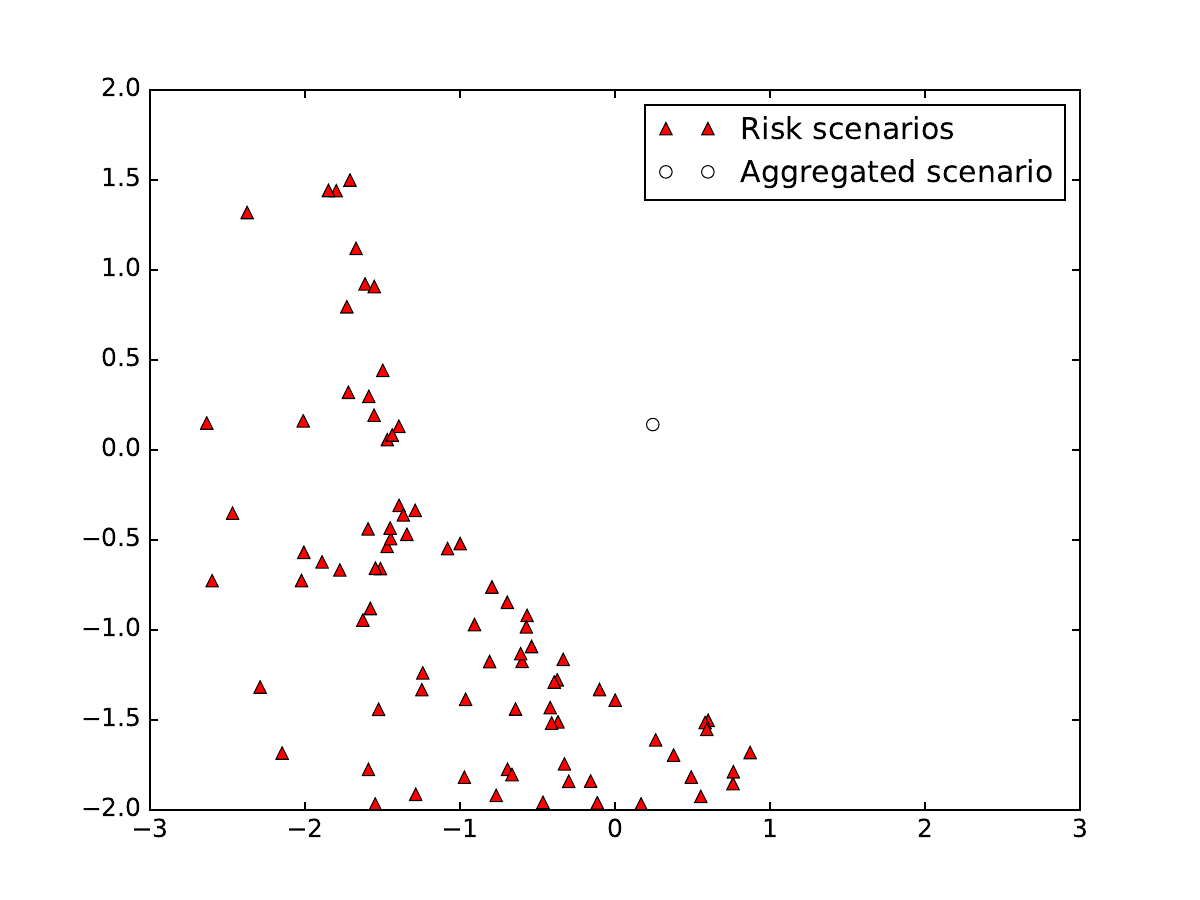}
    \caption{Scenario set constructed by aggregation reduction using conservative aggregation region}
    \label{fig:tailrisk-monotonic-scen-set}
  \end{subfigure}
  \begin{subfigure}[b]{0.45\textwidth}
    \includegraphics[width=\textwidth]{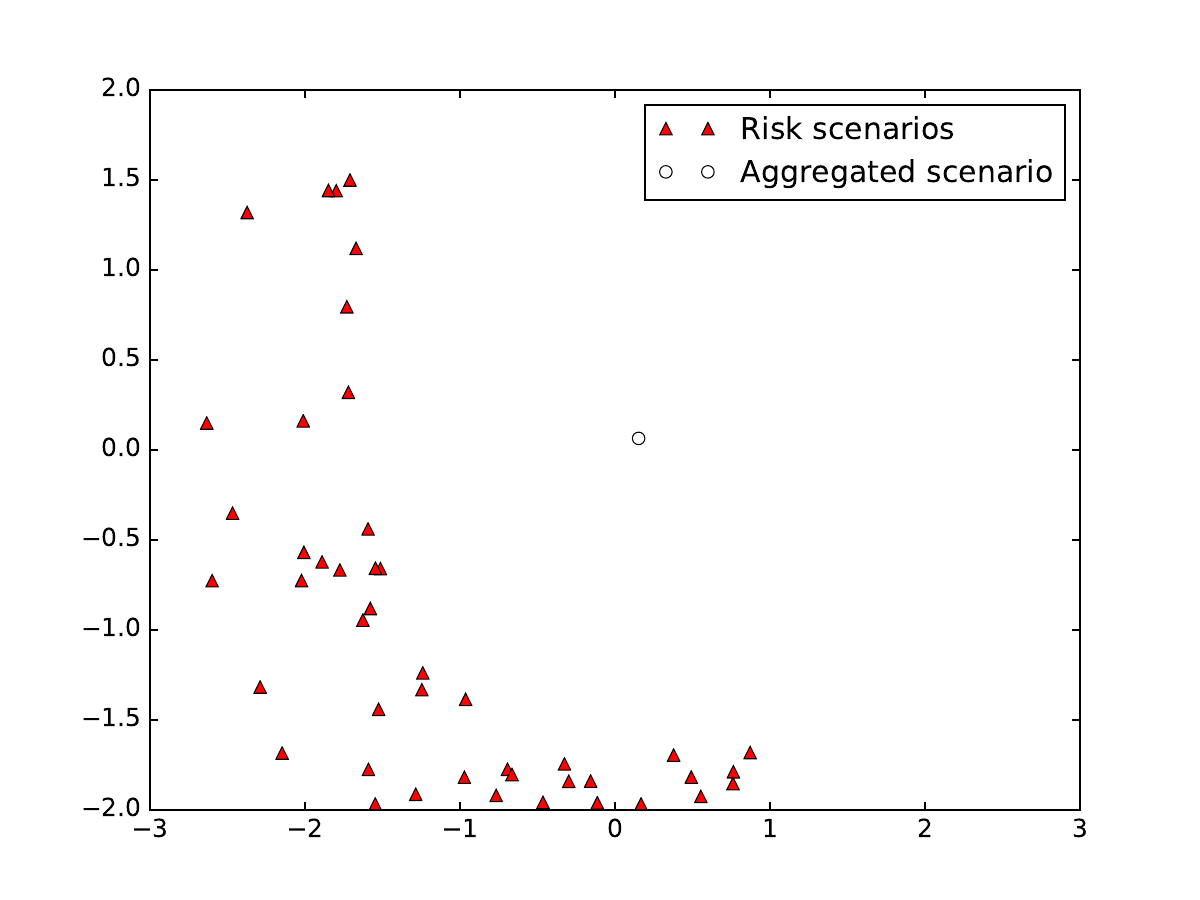}
    \caption{Scenario set constructed by aggregation reduction using exact aggregation region}
    \label{fig:tailrisk-nonrisk-scenset-ps}
  \end{subfigure}
  \caption{Probabilities of conservative and exact aggregation regions}
  \label{fig:tailrisk-probs-riskregion}
\end{figure}

\subsection{Performance of aggregation sampling}
\label{sec:tailrisk-numtests-perf-agg-sampling}

We now test the performance of the aggregation sampling
algorithm using conservative and exact risk regions against
standard Monte Carlo sampling in terms
of the quality of the solutions each method yields.

\paragraph{Experimental Set-up}

We use the following problem:
\begin{align*}
  \minimize[x\geq 0] &\cvar(-x^T\rv)\label{eq:p4}\tag{P}\\
  \text{subject to} \ x^T\mu &\geq t\nonumber\\
  \sum_{i=1}^{d}x_{i} &= 1\\
  x &\geq 0.
\end{align*}
where the asset returns follow a multivariate Normal distribution $\mathcal{N}(\mu, \Sigma)$.
We use two distributions: one of dimension 5 and another of dimension 10.
These distributions have been fitted from monthly return data for randomly selected companies
in the FTSE 100 index. The problem is thus to select a portfolio which minimizes the conditional
value-at-risk of the one-month return, subject to a minimimum expected return of $t$, and no short-selling.
These distributions have been made available online \cite{fairbrother17data}
in an HDF5 file, and can be accessed using the keys ``normal/dim = 5/dist 1'' and ``normal/dim = 10/dist 1''.
We use the target expected one-month return $t=0.005$ which is feasible for the constructed problems.

This problem has been chosen so that we can solve
the problem exactly for Normally distributed returns, and so calculate
the optimality gap for solutions found from solving scenario-based approximations.
The following formula
is easily verified by recalling that for continuous probability distributions,
the $\cvar$ is just the conditional expectation of the random variable
above the $\beta$-quantile (see \cite{Rockafellar00} for instance):
\begin{equation}
\label{eq:true-cvar}
\cvar(-x^T \rv) = (1 - \beta)\mu^T x + \sqrt{x^T\Sigma x} \int_{\Phi^{-1}(\beta)}^\infty z\ d\Phi(z)
\end{equation}
where $\Phi$ denotes the distribution function of the standard Normal distribution.
The problem \eqref{eq:p4} can therefore be solved exactly using an interior point
algorithm and in our experiments we use the software package IPOPT
\cite{WachterBiegler2006} to do this.

Denote by $\{(\orv_{s}, p_{s})\}_{s = 1}^{n}$ a scenario set of size $n$, where
$\orv_{s}$ denotes the vector of asset returns in scenario $s$, and 
$p_{s}$ the corresponding probability. Then, the scenario-based approximation
to \eqref{eq:p4} using this scenario set, is the following linear program:
\begin{align*}
  \minimize[x, y, \alpha] & \alpha + \frac{1}{1 - \beta} \sum_{s=1}^{n} p_{s}y_{s}\\
  \text{subject to } y_{s} & \geq -x^{T}\orv_{s} - \alpha\\
                             x^{T}\mu &\geq t\\
  \sum_{i=1}^{d} x_{i} &= 1\\
  x, y &\geq 0.
\end{align*}
See \cite{Rockafellar00} for more details on how $\cvar$ is linearized for discrete random vectors
in this way. These scenario-based problems are modelled using JuMP \cite{DunningHuchetteLubin2017}
and solved using Gurobi 7.5 \cite{gurobi}.

We are interested in the quality and stability of the solutions that
are yielded by our scenario generation method as compared to standard Monte Carlo
sampling for a given scenario set size. To this end, in each experiment, for a range of scenario set
sizes, we construct 100 scenario sets using sampling and aggregation
sampling with conservative and exact risk regions, solve the resulting
problems, and calculate the optimality gaps for the solutions that
these yield.

Denote by $z^{*}$ the optimal solution value for problem \eqref{eq:p4},
and by $\tilde{x}$ a solution found by solving a scenario-based approximation. 
Then the optimality gap of $\tilde{x}$ is given by
\begin{equation*}
  \cvar(-\tilde{x}^{T}\rv) - z^{*}
\end{equation*}
where $\cvar(-\tilde{x}^{T}\rv)$ calculated using \eqref{eq:true-cvar}.

\paragraph{Results}
\label{sec:tailrisk-numtests-results}

In Figure~\ref{fig:tailrisk-stability} are presented the results of
these stability tests for two different problems. In the first problem we have
$d = 5$ and $\beta=0.95$. In the second problem we have $d = 10$ and $\beta = 0.99$.
For each scenario set size and scenario generation method we have
drawn a box plot of the optimality gap of the 100 constructed scenario sets. 
In the legend of each plot we have given the estimated probability of
the aggregation regions, $a$, and the true optimal value $z^{*}$ is included
in the title. Note that Cons. Agg. sampling and Exact Agg. Sampling
are abbrieviations for, respectively, aggregation sampling using the conservative risk region, and aggregation
sampling using the exact risk region.

\begin{figure}[h]
  \centering
  \begin{subfigure}[b]{0.45\textwidth}
    \includegraphics[width=\textwidth]{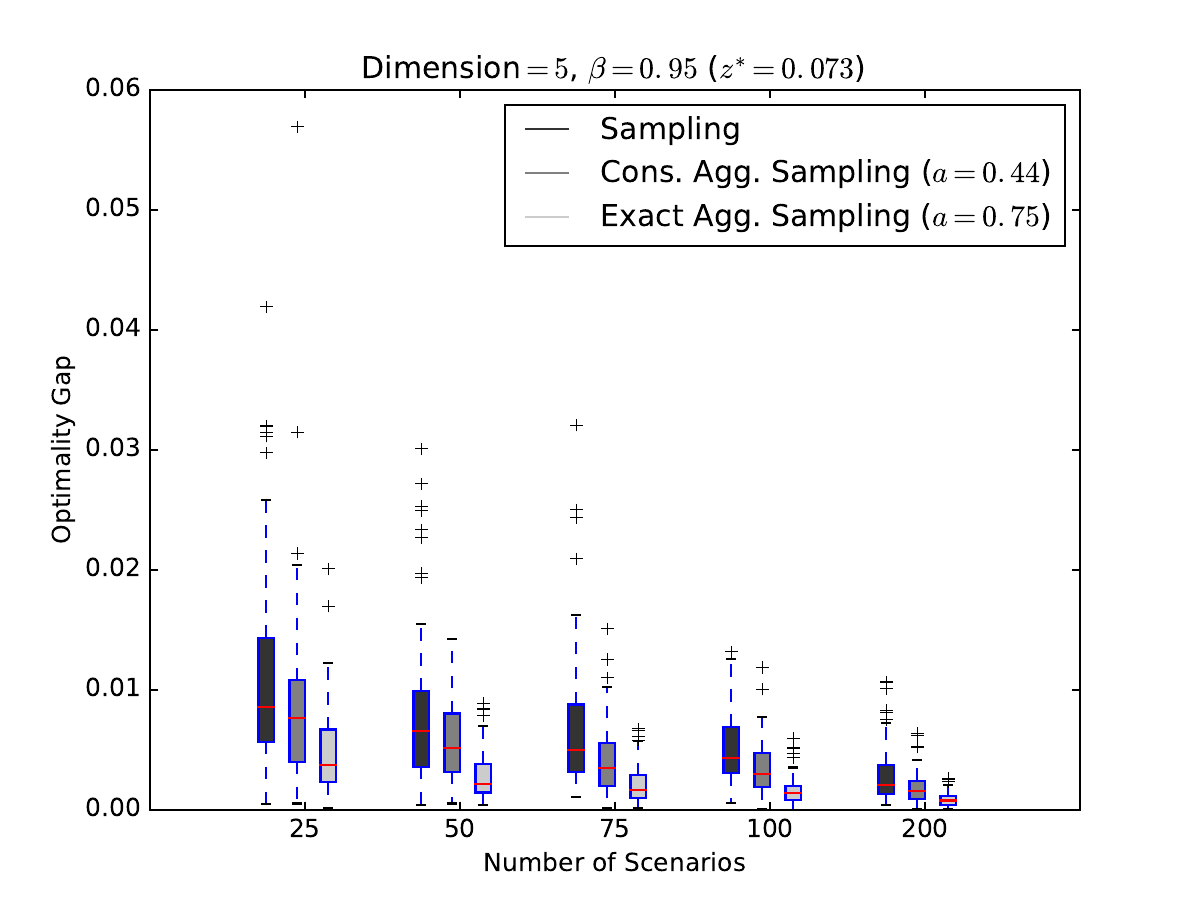}
  \end{subfigure}
  \begin{subfigure}[b]{0.45\textwidth}
    \includegraphics[width=\textwidth]{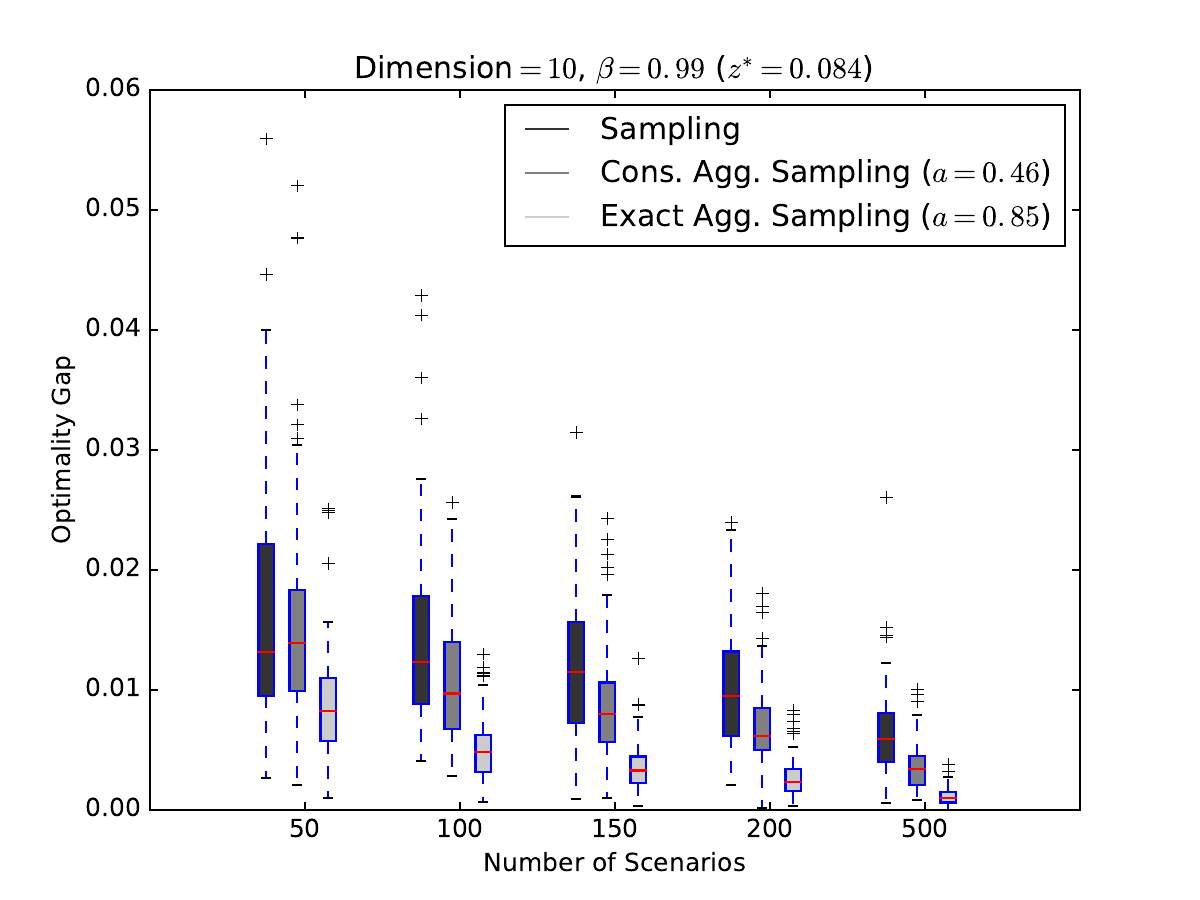}
  \end{subfigure}
  \caption{Error bar plot of optimality gaps yielded by sampling and aggregation sampling}
  \label{fig:tailrisk-stability}
\end{figure}

In both cases, both aggregation sampling methods outperform standard Monte Carlo sampling
for all scenario set sizes in terms of the size and variability of the
calculated optimality gaps. This is because for aggregation sampling we are
effectively sampling more scenarios compared with standard Monte Carlo sampling.
Aggregation sampling with exact risk regions also significantly outperforms 
aggregation sampling with conservative risk regions.
The improved performance can be expected given
that its probability is greater than that of the conservative risk region
which gives a greater effective sample size.

\section{Conclusions}
\label{sec:tailrisk-conclusions}

In this paper we have demonstrated that for stochastic programs which
use a tail risk measure, a significant portion of the support of the
random vector in the problem may not participate in
the calculation of that tail risk measure, whatever feasible decision
is used. As a consequence, for scenario-based problems, if we concentrate our scenarios in the region of the distribution which is important to the problem, the risk region,
we can represent the uncertainty in our problem in a more parsimonious way,
thus reducing the computational burden of solving it.

We have proposed and analyzed two specific methods of scenario generation using
risk regions: aggregation
sampling and aggregation reduction. Both of these methods were shown
to be more effective, in comparison to standard Monte Carlo sampling, as the probability of the non-risk region increases:
in essence the higher this probability the more redundancy there is
in the original distribution. 
The application of our methodology relies on having a 
convenient characterization of a risk region. 
For portfolio selection problems we derived the exact
risk region when returns have an elliptical distribution.
However, a characterization of the exact risk region will 
generally not be possible. Nevertheless, it is sufficient to 
have a conservative risk region. For stochastic programs with monotonic
loss functions, a wide problem class which includes some network
design problems, we were able to derive such a region.

The effectiveness of our methodology depends on
the probability of the aggregation region, that is
the exact or conservative non-risk region used in our scenario generation algorithms.
We observed that for both the stochastic programs with monotonic loss
function and portfolio selection problems
that this probability tends to zero as the dimension
of the random vector in the problem increases. However,
in some circumstances this effect is mitigated. 
We observed that small positive correlations slowed down this convergence
for the portfolio selection problem.

We tested the performance of our
aggregation sampling algorithm for portfolio selection problems
using both the exact non-risk region and the conservative risk
region for monotonic loss functions. This demonstrated a significant
improvement on the performance of standard Monte Carlo sampling, particularly
when an exact non-risk region was used.

The methodology has much potential. 
For some small to moderately-sized
network design problems this methodology could yield much better solutions.
In particular the methodology is agnostic to the presence of integer
variables, and so could be used to solve difficult mixed
integer programs.

In our follow-up paper \cite{Fairbrother15b}
we demonstrate that our methodology may be applied to more
difficult and realistic portfolio selection problems such as those involving integer variables, and for which the asset returns are no longer elliptically
distributed. In the same paper we also discuss some of the technical issues involved in
applying the method, such as finding the conic hull of the feasible region, and methods
of projecting points onto this. We also investigate the use of artificial constraints
as a way of making our methodology more effective.

\section*{Acknowledgements}
\label{sec:tailrisk-acknowledgements}

We would like to thank the reviewers and guest editor for their very thorough feedback which has allowed us to much improve this paper. Thanks also to Burak Buke and David Leslie who also gave feedback on an earlier version of the paper. Finally, we gratefully acknowledge the support of the EPSRC funded EP/H023151/1 STOR-i Centre for Doctoral Training.


\appendix

\section{Continuity of Distribution and Quantile Functions}
\label{sec:tailrisk-convergence-results}

Throughout we use the following set-up: $\mathcal{X} \subset \rk$ a
decision space, $\rv$ a random vector with support
$\rvsup\subset\rd$ defined on a probability space $(\Omega, \mathcal{B}, \mathbb{P})$, and a cost function $f:\mathcal{X}\times\rd \rightarrow \rr$.
The quantity is $f(x, \rv)$ is assumed to be measurable for all $x\in\mathcal{X}$.
In this appendix we prove a series of technical results
related to the continuity of the distribution and quantile
functions for $f(x,\rv)$. These are required for the proofs in
Section \ref{sec:tailrisk-conv-aggr-sampl}.

The following elementary result concerns the continuity of
an expectation function.

\begin{proposition}
  \label{prop:tailrisk-exp-cont}
  Suppose for $g:\mathcal{X}\times\rvsup\rightarrow \rr$, and a given $\bar{x}\in\mathcal{X}$ the following holds:
  \begin{enumerate}[(i)]
  \item $x\mapsto g(x,\rv)$ is continuous at $\bar{x}$ with probability 1,
  \item There exists a neighborhood $W$ of $\bar{x}$ and integrable $h:\rvsup \rightarrow \rr$ such that,
    for all $x\in W$ we have $g(x,\rv) \leq h(\rv)$ with probability 1.
  \end{enumerate}
  Then, $x\mapsto \E{g(x,\rv)}$ is continuous at $\bar{x}$.
\end{proposition}

\begin{proof}
  Let $(x_k)_{k=1}^\infty$ be some sequence in $\mathcal{X}$ such that $x_k\rightarrow\bar{x}$
  as $k\rightarrow\infty$. Without loss of generality $x_k\in W$ for
  all $k\in\mathbb{N}$.
  By assumption (i), almost surely we have $g(x_k, \rv) \rightarrow g(\bar{x}, \rv)$
  as $k\rightarrow\infty$. Using assumption (ii) we can apply
  the Lebesgue theorem of dominated convergence so that:
  \begin{align*}
    \lim_{k\rightarrow\infty} \E{g(x_k, \rv)} &= \E{\lim_{k\rightarrow\infty} g(x_k, \rv)}\\
    &=\E{g(\bar{x}, \rv)}
  \end{align*}
  and hence $x\mapsto\E{g(x,\rv)}$ is continuous at $\bar{x}$.
\qed\end{proof}

The continuity of the distribution function immediately follows
from the above proposition.

\begin{corollary}
  \label{cor:tailrisk-prob-cont}
  Suppose for a given $\bar{x}\in\mathcal{X}$ that $x\mapsto f(x,\rv)$ is continuous with probability 1
  at $\bar{x}$, and for $z\in\rr$ the distribution function $F_{\bar{x}}$ is continuous at $z$.
  Then, $x\mapsto F_{x}(z)$ is continuous at $\bar{x}$.
 \end{corollary}

 \begin{proof}
   Let $g(x,\rv) = \mathbbm{1}_{\{f(x,\rv) \leq z\}}$ so that $F_x(z) = \E{g(x,\rv)}$.
   The function $g(x,\rv)$ is clearly dominated by the integrable
   function $h(\rv) = 1$. It is therefore enough to show that $x\mapsto g(x,\rv)$ is almost surely
   continuous at $\bar{x}$ as the result will then follow from Proposition \ref{prop:tailrisk-exp-cont}.
   
   Since $F_{\bar{x}}$ is continuous at $z$, we must have
   $\Prob{f(\bar{x},\rv) = z} = 0$.  Almost surely, we have that for
   $\omega\in\Omega$ that $x\mapsto f(x,\rv(\omega))$ is continuous at $\bar{x}$. Let's
   first assume that $f(\bar{x},\rv(\omega)) > z$.  In this case, there exist some
   neighborhood $V$ of $\bar{x}$ such that $x\in V \Rightarrow
   f(x,\rv(\omega)) > z$, which in turn implies $\left| g(x,\rv) - g(\bar{x},
     \rv) \right| = 0$. 
   Hence $x\mapsto g(x,\rv(\omega))$ is continuous at
   $\bar{x}$. The same argument holds if $f(\bar{x},\rv(\omega)) < z$.  Hence,
   with probability 1, $x\mapsto g(x,\rv)$ is continuous at $\bar{x}$.
 \qed\end{proof}

Continuity of the quantile function follows from the continuity of
the distribution function but requires that the distribution
function is strictly increasing at the required quantile.

\begin{proposition}
  \label{prop:tailrisk-quant-cont}
  Suppose for some $\bar{x}\in\mathcal{X}$, and $z=F_{\bar{x}}^{-1}(\beta)$ that the conditions of
  Corollary \ref{cor:tailrisk-prob-cont} hold, and in addition
  that $F_{\bar{x}}$ is strictly increasing at $F_{\bar{x}}^{-1}(\beta)$, that is
  for all $\epsilon > 0$
    \begin{equation*}
      F_{\bar{x}}\left(F_{\bar{x}}^{-1}(\beta) - \epsilon\right) < \beta < F_{\bar{x}}\left(F_{\bar{x}}^{-1}(\beta) + \epsilon\right).
    \end{equation*}
  Then $x \mapsto F_x^{-1}(\beta)$ is continuous at $\bar{x}$.
\end{proposition}

\begin{proof}
  Assume $x\mapsto F_x^{-1}(\beta)$ is not continuous at $\bar{x}$. This
  means there exists $\epsilon > 0$ such that for all neighborhoods
  $W$ of $\bar{x}$
  \begin{equation*}
    \text{there exists } x' \in W \text{ such that } \left| F_{\bar{x}}^{-1}(\beta) - F_{x'}^{-1}(\beta)\right| > \epsilon.
  \end{equation*}
  Now set,
  \begin{equation*}
    \gamma := \min \{ \beta - F_{\bar{x}}\left(F_{\bar{x}}^{-1}(\beta) - \epsilon\right), F_{\bar{x}}\left(F_{\bar{x}}^{-1}(\beta) + \epsilon\right) - \beta\} > 0 \qquad \text{since } F_{\bar{x}} \text{ strictly increasing at } F_{\bar{x}}^{-1}\left(\beta\right).
  \end{equation*}
  By the continuity of $x \mapsto F_{x}\left(F_{\bar{x}}^{-1}(\beta)\right)$ at $\bar{x}$ there
  exists $W$ a neighborhood of $\bar{x}$, such that:
  \begin{equation}
    \label{eq:tailrisk-quantile-cont-ineq}
    x\in W \Longrightarrow \left| F_x\left(F_{\bar{x}}^{-1}(\beta)\right) - F_{\bar{x}}\left(F_{\bar{x}}^{-1}(\beta)\right)\right| < \gamma.
  \end{equation}
  But for the $x'$ identified above we have
  \begin{equation*}
    F_{x'}^{-1}(\beta) < F_{\bar{x}}^{-1}\left(\beta\right) - \epsilon \qquad
    \text{or} \qquad F_{x'}^{-1}(\beta) > F_{\bar{x}}^{-1}\left(\beta\right) + \epsilon
  \end{equation*}
  and so given that $F_{\bar{x}}$ is non-decreasing, and by the
  definition of $\gamma$ we must have:
  \begin{equation*}
    \left| F_{\bar{x}}\left(F_{\bar{x}}^{-1}(\beta)\right) - F_{\bar{x}}\left(F_{x'}^{-1}(\beta)\right) \right| \geq \gamma
  \end{equation*}
  which contradicts \eqref{eq:tailrisk-quantile-cont-ineq}.
\qed\end{proof}

Recall, that for a sequence of i.i.d. random vectors
$\rv_{1}, \rv_{2}, \ldots$ with the same distribution as $\rv$,
we define the sampled distribution function as follows:

\begin{equation*}
  F_{n,x}(z) := \frac{1}{n}\sum_{i=1}^n \mathbbm{1}_{\{f(x,\rv_i) \leq z\}}. 
\end{equation*}

The final result concerns the continuity of the sampled
distribution function.

\begin{lemma}
  \label{lem:tailrisk-conv-cont}
  Suppose for $g:\mathcal{X}\times\rvsup\rightarrow \rr$, and $\bar{x}\in\mathcal{X}$ the conditions
  from Proposition \ref{prop:tailrisk-exp-cont} hold.
  Then for all $\epsilon > 0$ there exists a neighborhood $W$, of $\bar{x}$, such that
  with probability 1  \begin{equation*}
    \limsup_{n\to\infty}\sup_{x\in W\cap\mathcal{X}}\left|\frac{1}{n}\sum_{i=1}^n g(x, \rv_i) - \frac{1}{n}\sum_{i=1}^n g(\bar{x}, \rv_i)\right| < \epsilon.
  \end{equation*}
  In particular, if $x\mapsto f(x,\rv)$ is continuous at $\bar{x}$ with
  probability 1 and $F_{\bar{x}}$ is continuous at $z\in\rr$ then for all $\epsilon>0$ there exists a neighborhood $W$,
  of $\bar{x}$ such that with probability 1
  \begin{equation}
    \limsup_{n\to\infty}\sup_{x\in W\cap\mathcal{X}} \left| F_{n,x}(z) - F_{n,\bar{x}}(z)\ \right| < \epsilon. \label{eq:tailrisk-cont-sampled-dist-fn-appendix}
  \end{equation}
\end{lemma}

\begin{proof}
  Fix $\bar{x}\in\mathcal{X}$, and $\epsilon > 0$. Let $(\gamma_k)_{k=1}^\infty$ be any sequence of positive
  numbers converging to zero and define 
  \begin{align*}
    V_k &:= \{ x\in \mathcal{X}: \norm[x - \bar{x}] \leq \gamma_k \},\\
    \delta_k(\rv) &:= \sup_{x\in V_k} \left| g(x,\rv) - g(\bar{x}, \rv)\ \right|.
  \end{align*}
  Note first that the quantity $\delta_k(\rv)$ is Lebesgue measurable
  (see \cite[Theorem~7.37]{ShapiroEA09} for instance).
  By assumption (i) of Proposition~\ref{prop:tailrisk-exp-cont} the mapping $x \mapsto g(x,\rv)$ is continuous at $\bar{x}$ with probability 1,
  hence $\delta_k(\rv) \rightarrow 0$ almost surely as $k \rightarrow \infty$. 
  Now, since $\left|g(x,\rv)\right| \leq h(\rv)$ we must have $|\delta_k(\rv)| \leq 2h(\rv)$, therefore,
  by the Lebesgue dominated convergence theorem, we have that
  \begin{equation}
    \label{eq:tailrisk-delta-limit}
    \lim_{k\rightarrow\infty} \E{\delta_k(\rv)} = \E{\lim_{k\rightarrow\infty}\ \delta_k(\rv)} = 0.
  \end{equation}
  Note also that
  \begin{equation*}
    \sup_{x\in V_k}\left|\frac{1}{n}\sum_{i=1}^n g(x, \rv_i) - \frac{1}{n}\sum_{i=1}^n g(\bar{x}, \rv_i)\right| \leq \frac{1}{n}\sum_{i=1}^n \sup_{x\in V_k}\left|g(x,\rv_i)- g(\bar{x}, \rv_i)\right|
  \end{equation*}
  and so
  \begin{equation*}
    \sup_{x\in V_k}\left|\frac{1}{n}\sum_{i=1}^n g(x, \rv_i) - \frac{1}{n}\sum_{i=1}^n g(\bar{x}, \rv_i)\right| \leq \frac{1}{n}\sum_{i=1}^n\delta_k(\rv_i).
  \end{equation*}
  Since the sequence of random vectors $\rv_1, \rv_2, \ldots$ is i.i.d. we have by the
  strong law of large numbers that the right-hand side of \eqref{eq:tailrisk-sample-dist-ineq}
  converges with probability 1 to $\E{\delta_k(\rv)}$ as $n\rightarrow\infty$. Hence, with probability 1
  \begin{equation}
    \label{eq:tailrisk-sample-dist-ineq}
    \limsup_{n\to\infty} \sup_{x\in V_{k}} \left| \frac{1}{n}\sum_{i=1}^{n} g(x, \rv_{i}) - \frac{1}{n} \sum_{i=1}^{n} g(\bar{x}, \rv_{i}) \right| \leq \E{\delta_{k}(\rv)}.
  \end{equation}
  By \eqref{eq:tailrisk-delta-limit}  we can pick $k\in\mathbb{N}$ such that $\E{\delta_{k}(\rv)} < \epsilon$ and so setting $W=V_{k}$ we have by
  \eqref{eq:tailrisk-sample-dist-ineq} with probability 1
  \begin{equation*}
    \limsup_{n\rightarrow\infty}\sup_{x\in W\cap\mathcal{X}} \left| \frac{1}{n}\sum_{i=1}^n g(x, \rv_i) - \frac{1}{n}\sum_{i=1}^n g(\bar{x}, \rv_i) \right| < \epsilon.
  \end{equation*}
  The result \eqref{eq:tailrisk-cont-sampled-dist-fn-appendix} follows immediately
  as the special case $g(x,\rv) = \mathbbm{1}_{\{f(x,\rv)\leq z\}}$.
\qed\end{proof}

\section{Convex cone results}
\label{sec:tailrisk-convex}

The results in this appendix relate to the characterization of the
non-risk region for the portfolio selection problem with
elliptically distributed returns.

The following two propositions give properties
about projections onto convex cones which are required in the proof of
the main results of this appendix.

\begin{proposition}
  \label{prop:convex-cone-proj}
  Suppose $K\subset\rd$ is a convex cone. Then, for all $\orv\in\rd$:
  \begin{equation*}
    p_K(\orv)^T \left(\orv - p_K(\orv)\right) = 0.
  \end{equation*}
\end{proposition}

\begin{proof}
  First note that we must have $p_{K}(\orv)^{T}\orv \geq 0$. If this is not the the case
  then
  \begin{equation*}
    \norm[\orv - p_{K}(\orv)]^{2} = \norm[p_{K}(\orv)]^{2} - 2 p_{K}(\orv)^{T} \orv + \norm[\orv]^{2}
    > \norm[\orv]^{2} = \norm[\orv - 0]^{2}
  \end{equation*}
  which contradicts the definition of $p_{K}(\orv)$ since $0\in K$.
  Now assume that $p_{K}(\orv)^{T}\left(\orv - p_{K}(\orv)\right) \neq 0$, and set
  $\tilde{x} = \frac{p_{K}(\orv)^{T}\orv}{\norm[p_{K}(\orv)]^{2}} p_{K}(\orv)\in K$. Now,
  \begin{equation*}
    p_{K}(\orv)^{T}(\tilde{x} - \orv) = p_{K}^{T}\orv - p_{K}^{T}\orv = 0.
  \end{equation*}
  By assumption $p_{k}^{T}\orv \neq \norm[p_{K}(\orv))]^{2}$, and so
  $\tilde{x} \neq p_{K}(\orv)$, hence
  \begin{align*}
    \norm[p_{K}(\orv) - \orv]^{2} &= \norm[(p_{K}(\orv) - \tilde{x}) + (\tilde{x} -\orv)]^{2}\\
    &= \norm[(p_{K}(\orv) - \tilde{x})]^{2}  - 2 \underbrace{(p_{K}(\orv)-\tilde{x})^{T}(\tilde{x} - \orv)}_{= 0} + \norm[(\tilde{x} - \orv)]^{2} 
    > \norm[(\tilde{x} - \orv)]^{2}
  \end{align*}
  which, again, contradictions the definition of $p_{K}(\orv)$ since $\tilde{x}\in K$.
\qed\end{proof}

\begin{proposition}
  \label{prop:proj-scalar-prod}
  Suppose $K\subset\rd$ be a convex cone and $x\in K$. Then for any $\orv\in\rd$
  \begin{equation*}
    x^T\orv \leq x^Tp_K(\orv).
  \end{equation*}
\end{proposition}

\begin{proof}
  The result holds trivially if $\orv\in K$ so we assume $\orv\notin K$.
  Assume there exists $\tilde{x}\in K$ such that $\tilde{x}^{T}\orv > \tilde{x}^{T}p_{K}(\orv)$.
  For all $0 \leq \lambda \leq 1$ we have $\lambda x + (1 - \lambda)p_{K}(\orv)\in K$.
  Now,
  \begin{align*}
    \norm[\left(\lambda\tilde{x} + (1-\lambda)p_{K}(\orv)\right) - \orv]^{2} - \norm[\orv -p_{K}(\orv)]^{2} &=\lambda^{2}\norm[\tilde{x} - p_{K}(\orv)]^{2} + 2\lambda(\tilde{x} - p_{K}(\orv))^{T}(p_{K}(\orv)-\orv) \\ 
    &= \lambda^{2}\norm[\tilde{x} - p_{K}(\orv)]^{2} - 2\lambda \underbrace{\tilde{x}^{T}(\orv -p_{K}(\orv))}_{> 0 \text{ by assumption}}.
  \end{align*}
  That is, for $0 < \lambda <  \frac{\tilde{x}^{T}(\orv - p_{K}(\orv))}{2\norm[p_{K}(\orv) - \tilde{x}]}$ we have 
  $\norm[\lambda\tilde{x} + (1-\lambda)p_{K}(\orv) - \orv] < \norm[\orv - p_{K}(\orv)]$ which contradicts
  the definition of $p_{K}(\orv)$.
\qed\end{proof}

The next two results describe the non-risk region for the
portfolio selection problem with elliptically distributed
returns when $\mathcal{X}$ is a convex set.
The first describes the exact non-risk region 
for elliptically distributed returns in the case $P = I$,
and the second generalizes the result to any non-singular matrix.

\begin{theorem}
  \label{thr:cone-proj-ellipse-1}
    Suppose $\mathcal{X}\subset\rd$ is convex and $\mu\in\rd$, and let $\mathcal{A} := \{\orv \in\rd: x^T(\orv -\mu) < \norm[x]\ \alpha\ \forall x \in \mathcal{X}\}$ and $\mathcal{B} := \{\orv\in\rd: \norm[p_K(\orv - \mu)] < \alpha\}$
  where $K = \conic{\mathcal{X}}$. Then, $\mathcal{A} = \mathcal{B}$.
\end{theorem}

\begin{proof}
  $(\mathcal{B} \subseteq \mathcal{A})$\\
  Suppose $\orv\in \mathcal{B}$ and let $x\in \mathcal{X}$, then $x\in K$ and so
\begin{align*}
  x^T (\orv - \mu) &\leq x^T p_K(\orv - \mu) \qquad \text{by Proposition \ref{prop:proj-scalar-prod}}\\
  &\leq \norm[x]\ \norm[p_K(\orv) - \mu] \qquad \text{by the Cauchy-Schwartz inequality}\\
  &< \norm[x]\ \alpha \qquad \text{since } \orv\in \mathcal{B}.
\end{align*}
  Hence $\orv\in \mathcal{A}$.\\
  $(\mathcal{A} \subseteq \mathcal{B})$\\
  Suppose $\orv\notin \mathcal{B}$ and set $x = p_K(\orv - \mu) \in K$. Now,
\begin{align*}
    x^T(\orv-\mu) &= p_K(\orv - \mu)^T(\orv - \mu)\\
    &= p_K(\orv - \mu)^T p_K(\orv - \mu) + p_{K}(\orv - \mu)^{T}\left((\orv-\mu) - p_{K}(\orv-\mu)\right)\\
    &= p_K(\orv-\mu)^T p_K(\orv-\mu) \qquad \text{by Proposition \ref{prop:convex-cone-proj}}\\
    &\geq \norm[x]\ \alpha \qquad \text{since } \orv\notin\mathcal{B}.
\end{align*}
Since $\mathcal{X}$ is convex we have $x = \lambda \bar{x}$ for some $\bar{x}\in\mathcal{X}$ and so 
we must also have $\bar{x}^{T}\orv \geq \norm[\bar{x}] \alpha$, hence $\orv\notin \mathcal{A}$.
\qed\end{proof}

\begin{corollary}
  \label{cor:cone-proj-ellipse-2}
  Suppose $\mathcal{X}$ is convex, and $P\in\rr^{d\times d}$ is a non-singular matrix. Let,
    $\mathcal{A} := \{\orv\in\rd: x^T(\orv-\mu) < \norm[Px]\ \alpha\ \forall x\in \mathcal{X}\}$
  and $\mathcal{B} := P^{T}\left(\{\tilde{\orv}\in\rd: \norm[p_{K'}(\tilde{\orv} - \tilde{\mu})] < \alpha\}\right)$
  where $\tilde{\mu} = (P^{T})^{-1}\mu$, $K' = PK$, and $K = \conic{\mathcal{X}}$. Then, $\mathcal{A} = \mathcal{B}$.
\end{corollary}

\begin{proof}
  First note that $K' = PK = P\conic{\mathcal{X}} = \conic{P\mathcal{X}}$. Now,
  \begin{align*}
    \mathcal{B} &= P^{T} \left( \{ \tilde{\orv}\in\rd: \norm[p_{K'}(\tilde{\orv} - \tilde{\mu})] < \alpha \}\right)\\
      &= P^{T} \left( \{ \tilde{\orv}\in\rd: \tilde{x}^T(\tilde{\orv}-\tilde{\mu}) < \norm[\tilde{x}]\ \alpha\ \forall \tilde{x}\in P\mathcal{X}\}\right) \qquad\text{by Theorem \ref{thr:cone-proj-ellipse-1}}\\
      &= \{\orv\in\rd: \tilde{x}^T\left((P^{T})^{-1}\orv -\tilde{\mu}\right)< \sqrt{\tilde{x}^T\tilde{x}} \alpha\ \forall \tilde{x}\in P\mathcal{X}\}\\
      &= \{\orv\in\rd: x^TP^{T}\left((P^{T})^{-1}\orv - (P^{T})^{-1}\mu\right) < \norm[Px]\ \alpha\ \forall x\in \mathcal{X}\}\\
      &= \{\orv\in\rd: x^T(\orv-\mu) < \norm[Px]\ \alpha\ \forall x\in K\} = \mathcal{A}
  \end{align*}
\qed\end{proof}


\bibliographystyle{plain}
\bibliography{bibtex-files/stein}

\end{document}